\documentclass{m2an_arxiv} % ARXIV
\usepackage[english]{babel}
\usepackage[colorlinks=true, linkcolor=blue, citecolor=blue, urlcolor=blue]{hyperref}
\usepackage{cleveref}
\usepackage{enumerate}
\usepackage{graphicx} 
\usepackage{subcaption}								% Enables subfigures
\usepackage{mathtools}
\usepackage{amsmath,bm}
\usepackage{amssymb}
\usepackage{amsfonts}
\usepackage{fourier} 

\usepackage{tikz-cd}

\numberwithin{equation}{section}					% Numbering order for equations
\numberwithin{table}{section}						% Numbering order for tables

\newcommand{\RR}{\mathbb{R}}

\DeclareMathOperator{\dof}{dof}

\DeclareMathOperator{\diag}{diag}

\DeclareMathOperator{\grad}{grad}
\DeclareMathOperator{\dive}{div}
\DeclareMathOperator{\crl}{curl}

\DeclareMathOperator{\add}{add}
\DeclareMathOperator{\mult}{mult}

\newcommand{\ekn}[1]{\textcolor{black}{#1}} % ARXIV

\newcommand{\mesh}{\mathcal{T}_h}

\newcommand{\mcA}{\mathcal{A}}

\newcommand{\mcF}{\mathcal{F}}
\newcommand{\mcE}{\mathcal{E}}

\newcommand{\mcN}{\mathcal{N}}

\newcommand{\Bf}{\mathcal{B}}

\newcommand{\nv}{\vec{n}}
\renewcommand{\nu}{\vec{n}}
\newcommand{\tv}{\vec{t}}
\newcommand{\xv}{\vec{x}}
\newcommand{\yv}{\vec{y}}

\newtheorem{assumption}{Assumption}
\newtheorem{remark}{Remark}
\newtheorem{theorem}{Theorem}
\newtheorem{lemma}[theorem]{Lemma}
\newtheorem{corollary}[theorem]{Corollary}

\numberwithin{assumption}{section}						% Numbering order for figures
\numberwithin{remark}{section}
\numberwithin{theorem}{section}						% Numbering order for figures
\begin{document} 

% \markboth{W.M. Boon and E. Nilsson}{Nodal Auxiliary Space Preconditioners for Mixed Virtual Element Methods}
% \runningtitle{Nodal Auxiliary Space Preconditioners for Mixed Virtual Element Methods}

\title[Nodal Auxiliary Space Preconditioners for Mixed Virtual Element Methods]{Nodal Auxiliary Space Preconditioners for Mixed Virtual Element Methods}
\thanks{WMB was supported by the European Union’s Horizon 2020 research and innovation programme under the Marie Skłodowska-Curie grant agreement No. 101031434 – MiDiROM. EN was supported by Swedish Research Council Grant No. 2018-05262 and the Wallenberg Academy Fellowship KAW 2019.0190.}

\author{Wietse M. Boon}
\address{NORCE Norwegian Research Centre, N-5838 Bergen, Norway, E-mail: wibo@norceresearch.no}

\author{Erik Nilsson}
\address{Department of Mathematics, KTH Royal Institute of Technology,  SE-100 44 Stockholm, Sweden. E-mail: erikni6@kth.se}
%\email{erikni6@kth.se}

\begin{abstract}
	We propose nodal auxiliary space preconditioners for facet and edge virtual elements of lowest order by deriving discrete regular decompositions on polytopal grids and generalizing the Hiptmair-Xu preconditioner to the virtual element framework. The preconditioner consists of solving a sequence of elliptic problems on the nodal virtual element space, combined with appropriate smoother steps. Under assumed regularity of the mesh, the preconditioned system is proven to have bounded spectral condition number independent of the mesh size and this is verified by numerical experiments on a sequence of polygonal meshes. Moreover, we observe numerically that the preconditioner is robust on meshes containing elements with high aspect ratios.
\end{abstract}


\maketitle

\section{Introduction}\label{sec: intro}

The Virtual Element Method (VEM) is a discretization method that handles polytopal grids without requiring explicit knowledge of the local basis functions. The method was initially introduced as a generalization of the Finite Element Method, leading to discrete spaces that are proper subspaces of the Sobolev space $H^1$ \cite{Hitchhiker14,beirao2013basic}. Subsequently, well-known \emph{mixed} finite elements were generalized to the VEM framework and discrete subspaces of $H^{\dive}$ and $H^{\crl}$ have been proposed, with applications in porous media flow and electromagnetism \cite{VeiBeiBrezzMarRuss2015,Brezzi2014VEMMixedPrinciples,DaVeiga2022Maxwell,VeiBeiBrezzMarRuss16}.

Similar to (mixed) finite element methods, the condition number of the linear system obtained from VEM depends on the quality of the mesh, in particular on the mesh size and the aspect ratios of its elements. This poses a problem for iterative solvers in case the problem is too large for direct solvers to handle. The performance of these solvers is directly affected by the ill conditioning of the system matrix and smaller mesh sizes may lead to larger numbers of iterations before convergence. The main objective of this work is to construct preconditioners for systems discretized by mixed virtual element methods that are robust with respect to the mesh size. 

We follow the framework of auxiliary space preconditioning \cite{xu1996auxiliary}, \ekn{which has been shown to be an effective method for $H^{\dive}$ and $H^{\crl}$ problems \cite{Kolev2009,Kolev2012}.} In particular we formulate a generalization of the Hiptmair-Xu preconditioner \cite{HiptXu07} for mixed virtual element methods of lowest order. The main advantage of this approach is that the original problem, posed on facet or edge spaces, is replaced by a series of elliptic problems posed on the nodal virtual space. The nodal space is the most well understood virtual element space and several efficient iterative solvers for such problems are available in the literature \cite{antonietti2018multigrid,bertoluzza2017bddc,calvo2018approximation}.

It is important to note that the theoretical results presented herein only cover regular meshes containing convex elements with mild aspect ratios. However, the Hiptmair-Xu preconditioner in the case of finite elements often performs robustly in the presence of badly shaped elements as well \cite{budisa2020mixed}. If this behavior is inherited to the virtual element case, then the proposed preconditioner will enable us to handle a wider class of polyhedral grids. In this context, we are interested in grids that are adapted from a background mesh to conform to interfaces or features that are arbitrarily placed in the computational domain. Such grid adaptations often result in arbitrarily shaped polygons and a robust preconditioner may be key in handling the corresponding, ill-conditioned systems. We investigate these possibilities numerically in \Cref{sec:numex}.

We note that the virtual element method can handle elements with high aspect ratios by augmenting the stabilization term, see e.g. \cite{mascotto2023VEMstabilizationReview} for an overview on this strategy. In contrast, our preconditioning approach leaves the discretization method intact and is therefore less intrusive.
We furthermore mention that parallel solvers for mixed virtual element methods were previously proposed in \cite{dassi2020parallel}. In comparison to that work, our proposed preconditioners have a strong theoretical foundation, which allows us to rigorously prove robustness with respect to the mesh size.
Finally, auxiliary space preconditioners for linear, nodal VEM were introduced in \cite{zhu2020auxiliary}. The difference with that work is that we propose a construction based on a regular decomposition, instead of a simplicial subtesselation of the mesh.

The main contributions of this work are summarized as follows:
\begin{itemize}
	\item We derive stable, regular decompositions of the facet and edge virtual element spaces of lowest order.
	\item We propose a Hiptmair-Xu preconditioner for elliptic problems that consists of continuous transfer operators, continuous smoothers, and a sequence of elliptic problems on the nodal virtual element space. The condition number of the preconditioned system is proven to be independent of the mesh size.
	\item Numerical experiments in 2D validate the theoretical results and moreover indicate robust performance in the presence of elements with high aspect ratios.
\end{itemize}

Moreover, the following observations serve as additional contributions. First, we use the regular decomposition to provide an alternative proof that the lowest order conforming VE spaces form an exact co-chain complex in \Cref{thm:discrete exact complex}. Second, we observe in \Cref{sub: smoother} that the stabilization term that is common to virtual element methods has the sufficient properties to serve as a suitable smoother. Third, we used the smoothers to derive inverse inequalities for the lowest order facet and edge virtual element spaces in \Cref{cor: inverse inequalities}.

This article is organized as follows. First, \Cref{sec: Preliminaries} introduces the relevant concepts from Sobolev theory, differential complexes, and VEM. In \Cref{sec: VEM Decomposition}, we place the lowest order VE spaces in the context of co-chain complexes and prove the discrete regular decomposition. \ekn{We note that the sections preceding \Cref{sec:cochain} primarily concern known results, whereas our main contributions are presented afterward.} The framework of auxiliary space preconditioning is introduced in \Cref{sec:auxprecond}; we construct the preconditioner and prove its robustness. \Cref{sec:numex} contains the numerical experiments that validate the theory and showcase the performance of the preconditioner. Conclusions are presented in \Cref{sec:conclusion}.

\section{Preliminaries and notation}
\label{sec: Preliminaries}

In this section, we introduce the notation conventions and present several key results from functional analysis and VE theory. 
The definitions concerning the relevant function spaces are summarized at the end of the section in \Cref{tab:spaces_summary}, for convenience.

\subsection{Continuous function spaces} \label{sub: continuous spaces}

Let $\Omega\subset \RR^n$ with $n\in\{2,3\}$ be a bounded, contractible domain with Lipschitz boundary $\partial\Omega$. Let $\nv$ denote the outward oriented unit vector that is normal to $\partial\Omega$. 
$L^2(\Omega)$ denotes the Hilbert space consisting of square-integrable functions on $\Omega$, endowed with inner product $(\cdot,\cdot)$ and norm $\|\cdot\|$. 
With a slight abuse of notation, we reuse $(\cdot,\cdot)$ and $\|\cdot\|$ for analogous inner product and norm on vector-valued distributions.
We moreover define:
\begin{subequations}
\begin{align}
H^{1}(\Omega) &:= \left\{v \in L^2(\Omega) : \ \grad v \in [L^2(\Omega)]^n\right\}, &
H^{\dive}(\Omega) &:= \left\{v \in  [L^2(\Omega)]^n : \ \dive v \in L^2(\Omega)\right\}, \\
H^{\crl}(\Omega) &:= \begin{cases}
	\left\{v \in  [L^2(\Omega)]^3 : \ \crl v \in [L^2(\Omega)]^3\right\},\ &n=3, \\
	H^1(\Omega),\ &n=2.
\end{cases}
\end{align}
\end{subequations}

% These functions spaces are endowed with the following norms:
% \begin{align}
% \| v \|_{1}^2 &:= \| v \|^2+\| \grad p \|^2, &
% \| v \|_{\dive}^2 &:= \| v \|^2+\| \dive v \|^2, &
% \| v \|_{\crl}^2 &:= \| v \|^2+\| \crl v \|^2. 
% \end{align}
% When considering a subset $U\subseteq\Omega$ we will instead write $\|v\|_U$ for the $L^2$-norm on $U$, and
% \begin{align}
% 	\| v \|_{1,U}^2 &:= \| v \|_U^2+\| \grad p \|_U^2, &
% 	\| v \|_{\dive,U}^2 &:= \| v \|_U^2+\| \dive v \|_U^2, &
% 	\| v \|_{\crl,U}^2 &:= \| v \|_U^2+\| \crl v \|_U^2. 
% \end{align}

The subspaces relevant for problems with essential boundary conditions are defined as follows:
\begin{subequations} \label{eqs: H0 spaces}
	\begin{align}
		H^{1}_0(\Omega) &:= \left\{v \in H^1(\Omega) : \ v|_{\partial\Omega}=0 \right\}, &
		H_0^{\dive}(\Omega) &:= \left\{v \in  H^{\dive}(\Omega) : \ v\cdot \nv |_{\partial \Omega}=0\right\}, \\
		H_0^{\crl}(\Omega) &:= \begin{cases}
			\left\{v \in  H^{\crl}(\Omega) : \ \nv\times v |_{\partial\Omega}=0\right\},\ &n=3, \\
			H_0^1(\Omega),\ &n=2.
		\end{cases}& 
		L^2(\Omega) / \RR &:= \left\{v \in L^2(\Omega) : \ \int_{\Omega} v=0 \right\}.
	\end{align}
\end{subequations}

\begin{remark}[The two-dimensional $\crl$]
	While we aim to employ a unified treatment with respect to dimension, the definition of the $\crl$ in 2D requires some additional details. 
    For a scalar field $p$ and a vector field $v=(v_1,v_2)$ we define, respectively, 
    \begin{subequations}
    \begin{align} 
        \crl p &:= \grad^{\perp} p = (-\partial_2 p,\partial_1 p), &
        \crl v &:= \dive v^{\perp} = \partial_2 v_1-\partial_1v_2,
        % \nv\times v &:= \nv^{\perp} \cdot v = n_2v_1-n_1v_2
        \end{align}
    \end{subequations}
    in which the superscript $\perp$ denotes a counter clockwise rotation of a vector by $\pi/2$.
\end{remark}

\subsection{Model problems}
\label{sub: Model problems}

The problems of interest are the projection problems onto the Sobolev spaces introduced in \Cref{sub: continuous spaces}. Namely, given $f\in [L^2(\Omega)]^n$,
\begin{align}
	\text{Find } u &\in H^{\dive}(\Omega) : & (u,v) + (\dive u,\dive v) &= (f,v), &\forall v&\in H^{\dive}(\Omega),\label{eq:divdiv} \\
	\text{Find } u &\in H^{\crl}(\Omega) : & (u,v) + (\crl u,\crl v) &= (f,v), &\forall v&\in H^{\crl}(\Omega).
\end{align}

For the numerical validation of the preconditioner, we will in addition to \eqref{eq:divdiv} consider the mixed formulation of the Poisson problem, also known as the Darcy problem. In weak form it reads as follows: given $f \in L^2(\Omega)^n$ and $g \in L^2(\Omega)$, find $u\in H^{\dive}(\Omega)$ and $p\in L^2(\Omega)$ such that
\begin{subequations}\label{eq:darcy}
\begin{align}
	(u,v) - (p,\dive v) &= (f,v), &\forall v&\in H^{\dive}(\Omega), \\
	-(q,\dive u) &= (g,q), &\forall q&\in L^2(\Omega).
\end{align}
\end{subequations}
% Here $g\in L^2(\Omega)$ is a given scalar source term.

\subsection{Exact co-chain complexes}\label{sec:cochaincomplexes}

A co-chain complex $(V^\bullet, d^\bullet)$ is a sequence of linear spaces $V^k$ and differential maps $d^k:V^k\to V^{k+1}$ with the property $d^{k + 1}d^k = 0$ for all $k$. We will omit the superscript on $d$ for notational brevity when $k$ is clear from context. A co-chain complex can be illustrated as the following sequence,
in which two consecutive steps in the diagram map to zero:
\begin{align} \label{eq: generalcomplex}
	\dots \rightarrow V^{k - 1} \overset{d}{\to} V^k \overset{d}{\to} V^{k + 1} \to \ldots
\end{align}

We will focus on complexes that only have non-trivial $V^k$ for $0 \le k \le n$ with $n$ the dimension of the domain $\Omega$. 
For $n = 3$, we thus consider a specific instance of \eqref{eq: generalcomplex} given by:
\begin{align} \label{eq:derhamcomplex}
    0 \rightarrow V^{n-3} \overset{d}{\to} V^{n-2} \overset{d}{\to} V^{n-1} \overset{d}{\to} V^n \to 0.
\end{align}
In the following, we will thus have $V^k = 0$ and $d^k = 0$ for all $k \not \in [0, n]$.

Since $dd=0$ by definition, we have $dV^k \subseteq \ker_d V^{k+1}$ in which $\ker_d V^{k+1}$ denotes the null space $\{v \in V^{k+1} : dv = 0\}$.
We refer to the co-chain complex as \emph{exact} if the converse inclusion holds as well, i.e. if
\begin{align}
	dV^k = \ker_d V^{k+1},
\end{align}

In this work, we focus on spaces $V^k$ that are subsets of $L^2(\Omega)$ and we therefore endow each $V^k$ with the following graph norm:
\begin{align} \label{eq: graph norm}
	\|v\|^2_{V^k} := \|v\|^2+\|dv\|^2.
\end{align} 

The following important result follows directly from the exactness of a co-chain complex. 
\begin{lemma}[{Poincaré inequality \cite[Thm. 4.6]{ArnFEEC18}}] \label{lem: Poincare}
	If a co-chain complex $(V^\bullet, d)$ is exact, then:
    \begin{align}
        \|v\| &\lesssim \|dv\|, &
        \forall v &\perp \ker_d V^k.
    \end{align}
\end{lemma}

A consequence of \Cref{lem: Poincare} is the following result.
\begin{lemma}[Stable potentials]\label{lem:stable_pot}
    Let $(V^\bullet, d)$ be an exact co-chain complex. Then, for $v\in V^{k}$ with $dv=0$, there exists a potential $u\in V^{k-1}$ that satisfies
    \begin{align}
        du &= v, &
        \|u\|_{V^{k - 1}} &\lesssim \|v\|.
    \end{align}
\end{lemma}
\begin{proof}
    First, by exactness, there exists a $\bar{u}\in V^{k-1}$ for which $d\bar{u} = v.$
    Next, we note that the space $V^{k - 1}$ can be decomposed as
    \begin{align}
        V^{k-1}=\ker_d(V^{k-1})\bigoplus \ker_d(V^{k-1})^{\perp}
    \end{align}
    and thus we may write $\bar{u}=z+u$ with $z\in\ker_d(V^{k-1})$ and $u\perp \ker_d(V^{k-1})$. In turn, it follows that 
    \begin{align}
        du=d(\bar{u}-z)=d\bar{u}=v.
    \end{align}
    Using once more that the complex is exact, we invoke \Cref{lem: Poincare} on $u\perp \ker_d(V^{k-1})$ to obtain the result. 
    (Since $du=v$, it is clear that we also have $\|du\|\lesssim \|v\|$.) 
\end{proof}

We focus on a specific co-chain complex, known as the de Rham complex on $\Omega$. The two- and three-dimensional cases are as follows:
\begin{subequations}\label{eqs:natural_derhamcomplex}
\begin{align}
    0 \hookrightarrow H^1(\Omega) / \RR \overset{\grad}{\to} H^{\crl}(\Omega) \overset{\crl}{\to} H^{\dive}(\Omega) \overset{\dive}{\to} L^2(\Omega) &\to 0,& n=3, \label{eq:3Dsequence} \\
    0 \hookrightarrow H^1(\Omega) / \RR \overset{\crl}{\to} H^{\dive}(\Omega) \overset{\dive}{\to} L^2(\Omega) &\to 0,& n=2. \label{eq:2Dsequence} 
\end{align}
\end{subequations}
These complexes are appropriate for handling natural boundary conditions on $\partial \Omega$. If essential boundary conditions are needed, we use the spaces $H_0^\bullet$ from \eqref{eqs: H0 spaces} and consider the complexes:
\begin{subequations}\label{eqs:BC_derhamcomplex}
\begin{align}
    0 \hookrightarrow H_0^1(\Omega) \overset{\grad}{\to} H_0^{\crl}(\Omega) \overset{\crl}{\to} H_0^{\dive}(\Omega) \overset{\dive}{\to} L^2(\Omega) / \RR &\to 0,& n=3, \\
    0 \hookrightarrow H_0^1(\Omega) \overset{\crl}{\to} H_0^{\dive}(\Omega) \overset{\dive}{\to} L^2(\Omega) / \RR &\to 0,& n=2. 
\end{align}
\end{subequations}

From now on, we define $V^k$ and $d$ such that \eqref{eq:derhamcomplex} corresponds to one of these four complexes, each of which is exact since $\Omega$ is contractible \cite{ArnFEEC18}. More precisely, we define $V^n$ as the rightmost $L^2$ space and $V^{n - 1}$ will be the $H^{\dive}$ to its left. The notation $V^{n-2}$ will only be used for $n=3$. Finally, $V^0$ will refer to the $H^1$ space in the left of the complex. A summary of $V^k$ and all other function spaces introduced in this section is provided in \Cref{tab:spaces_summary} at the end of the section.

We emphasize that the graph norm $\| \cdot \|_{V^k}$ from \eqref{eq: graph norm} coincides with the norms conventionally assigned to the corresponding Sobolev spaces. For example, we have $\| v \|_{V^{n - 1}}^2 = \| v \|^2 + \| \dive v \|^2 =: \| v \|_{H^{\dive}}^2$ and $\| p \|_{V^n} = \| p \|$. Moreover, the projection problems from \Cref{sub: Model problems} can be rewritten as: Given $f$, find $u \in V^k$ such that
\begin{align}
	(u, v) + (du, dv) &= (f, v), &
	\forall v &\in V^k.
\end{align}

\subsection{Regular decomposition}
\label{sub: Regular decomposition}

A key tool in the construction of nodal auxiliary space preconditioners is the existence of a decomposition into functions of higher regularity. For that, we introduce the following subspace $W^k\subseteq V^k$ for each $k$:
\begin{align}\label{eq:Wk}
	W^k := \left[H^1(\Omega)\right]^{\binom{n}{k}} \cap V^k.
\end{align}
We emphasize that $W^n, W^0 \subseteq H^1(\Omega)$ while $W^k \subseteq [H^1(\Omega)]^{n}$ for $0<k<n$, if $n = 2,3$. The intersection with $V^k$ ensures that the functions in $W^k$ satisfy the boundary conditions in the case of complex \eqref{eqs:BC_derhamcomplex}. The regular decomposition of $V^k$ is presented in the following lemma.

\begin{lemma}[{Continuous regular decomposition \cite[Lem. 3.10]{HiptXu07}}]\label{lem:cont_decomp}
	Given $v\in V^k$, then $\psi\in W^k$ and $p\in V^{k-1}$ exist such that
	\begin{align}
		v&=\psi+dp, &
		\|\psi\|_1&\lesssim \|dv\|, & 
		\|p\|_{V^{k - 1}} &\lesssim \|v\|_{V^k}.
	\end{align}
\end{lemma}

\subsection{Virtual element spaces of lowest order}\label{sec:vemintro}
We will now introduce the necessary discrete concepts. Let $\mesh$ be a polytopal tesselation of $\Omega$.
For the geometry we use the following nomenclature. An element $K$ is $n$-dimensional, a facet is $(n-1)$-dimensional, an edge is of dimension $1$, and a node is $0$-dimensional. For $K$ an element of $\mesh$, let $\mcF(K)$ be the set of facets that compose $\partial K$, $\mcE(K)$ the set of the corresponding edges, and $\mcN(K)$ the set of nodes neighboring $K$.
The sets of all facets, edges, and nodes of $\mesh$ are then defined as: 
\begin{align}
    \mcF &:=\bigcup_{K\in\mesh}\mcF(K), &
    \mcE &:=\bigcup_{K\in\mesh}\mcE(K),  &
    \mcN &:=\bigcup_{K\in\mesh}\mcN(K). 
\end{align}
To each facet $F\in\mcF$ we associate a unique normal vector $\nv$ and to each edge $E \in \mcE$ a unit tangential vector $\tv$.

\begin{assumption}%[{\cite{DaVeiga2022LowestInterpop,DaVeiga2022Maxwell}}]
\label{ass:mesh}
	The polytopal mesh $\mesh$ satisfies the following two conditions:
	\begin{enumerate}
		\item There exists a constant $h > 0$ such that each mesh entity $\sigma$ of dimension $n_\sigma \ge 1$ has diameter $\operatorname{diam}(\sigma) \eqsim h$ and measure $|\sigma| \eqsim h^{n_\sigma}$.
		\item Every element $K\in\mesh$ and every facet $F\in\mcF(K)$ is a strictly convex polytope. \label{enum:ass_4}
	\end{enumerate}
\end{assumption}

The relation $a\lesssim b$ implies that a constant $C > 0$ exists, independent of the mesh size $h$, such that $C a\leq b$. The converse $a\gtrsim b$ is defined similarly and we denote $a \eqsim b$ if $a \lesssim b \lesssim a$.

Let us continue by considering the local virtual element spaces $V_h^k(K)$ with $0 \le k \le n$ and $K \in \mesh$ an element. We start by defining $V_h^n(K)$ as the space of constants, with $\dof^n_K$ the associated degree of freedom:
\begin{align} \label{eq: def V_h^n}
	V_h^n(K) &:= \mathbb{P}_{0}(K), &
	\dof^n_K(v) &:= \int_K v.
\end{align}

The next local VE spaces on $K$ (cf. \cite{DaVeiga2018Lowest,DaVeiga2022LowestInterpop,DaVeiga2022Maxwell}) are of finite dimension, but contain functions that are not necessarily polynomials. Let the facet VE space $V_h^{n - 1}(K)$ that generalizes the Raviart-Thomas element \cite{RaviartThomas77} be given by
\begin{subequations} \label{eqs: def facet space}
\begin{align}
	V_h^{n-1}(K) := \ &\big\{v\in [L^2(K)]^2: \dive v \in \mathbb{P}_{0}(K),\ \crl v \in \mathbb{P}_{0}(K), \nonumber\\ 
	&(v\cdot \nv)|_F \in \mathbb{P}_0(F) \ \forall F\in\mcF(K),\ \int_K v \cdot \xv_K^{\perp}=0 \big\}, & 
	n&=2,\\
	V_h^{n-1}(K) := \ &\big\{v\in [L^2(K)]^3: \dive v \in \mathbb{P}_{0}(K),\ \crl v \in [\mathbb{P}_{0}(K)]^3, \nonumber\\ 
	&(v\cdot \nv)|_F \in \mathbb{P}_0(F) \ \forall F\in\mcF(K),\ \int_K  v \cdot (\xv_K\times \yv)=0 \ \forall \yv\in[\mathbb{P}_0(K)]^3 \big\}, & 
    n&=3,
\end{align} 
in which $\xv_K:=\xv-\vec{c}_K$ with $\vec{c}_K$ the centroid of $K$. This space has one degree of freedom on each facet, given by
\begin{align}\label{eq: def dof facet}
    \dof_F^{n-1}(v) &:= \int_F v\cdot \nv, &
    \forall F &\in \mcF.
\end{align}
\end{subequations}

Next, we introduce the edge VE space for $n = 3$, which is similar to the Nédélec edge element of the first kind \cite{Nedelec80}:
\begin{subequations} \label{eqs: def edge space}
\begin{align}
	V_h^{n-2}(K) := \ &\big\{v\in [L^2(K)]^3: \dive v =0,\ \crl\crl v \in [\mathbb{P}_{0}(K)]^3, \nonumber\\ 
	&(\nv\times v)|_F \in V_h^{n-2}(F) \ \forall F\in\mcF(K),\nonumber\\
	&(v|_F \cdot \tv)|_E = ( v|_{F'}\cdot \tv)|_E \ \forall E \in \mcE(F) \cap \mcE(F') \text{ with } F,F' \in \mcF(K), \nonumber\\
	& \int_K  \crl v \cdot (\xv_K\times \yv)=0 \ \forall \yv\in[\mathbb{P}_0(K)]^3 \big\}
\end{align}
In this case, the degrees of freedom are defined on the edges of the mesh:
\begin{align}\label{eq: def dof edge}
    \dof_E^{n-2}(v) &:= \int_E v\cdot \tv, &
    \forall E &\in \mcE(K).
\end{align}
\end{subequations}

\ekn{Lastly, the VE nodal space, which generalizes the Lagrange finite element, is given by}
\begin{subequations} \label{eqs: def nodal space}
	\begin{align} 
		V_h^{0}(K) &:= \big\{v\in C^0(K): \Delta v \in \mathbb{P}_{0}(K),\ v|_F \in \mathbb{P}_1(F) \ \forall F\in\mcF(K),\ \int_K \grad v \cdot \xv_K=0 \big\},& 
        n&=2,\\
		V_h^{0}(K) &:= \left\{v\in C^0(K): \Delta v = 0,\ v|_F \in V_h^0(F) \ \forall F\in\mcF(K) \right\},
        & n&=3,
	\end{align} 
	endowed with the degrees of freedom:
	\begin{align}
		\dof_{\xv}^0(v) &:= v(\xv), &
		\forall \xv &\in\mcN(K).
	\end{align}
\end{subequations}

Let us remark on a convenient property that all $V_h^k(K)$ share, namely that they are each locally $H^1$-regular.
\begin{lemma}[Local $H^1$-regularity]\label{lem:subspaces_H1}
    For each $K\in\mesh$, % and $0 \le k \le n$,
	it holds that $V_h^k(K)\subset W^k(K)$.
\end{lemma}
\begin{proof}
    We first show that $\Delta v \in \mathbb{P}_0(K)$ for all $v \in V_h^k(K)$. We distinguish the four cases:
    \begin{itemize}
    	\item ($k = n$) The definition $v\in V_h^n(K) := \mathbb{P}_0(K)$ directly gives us $\Delta v = 0$.
    	\item ($k = n - 1$) Recall the following calculus identity, which holds for $n = 2, 3$:
	    \begin{align} \label{eq: vec Laplacian}
	        \Delta v &= -\crl(\crl v)+\grad(\dive v).
	    \end{align}
    	From \eqref{eqs: def facet space}, we immediately obtain $\Delta v = 0$.
    	\item ($k = n - 2$, $n = 3$) Using \eqref{eq: vec Laplacian}, the properties \eqref{eqs: def edge space} gives us $\Delta v \in \mathbb{P}_0(K)$.
    	\item ($k = 0$) The definitions \eqref{eqs: def nodal space} imply $\Delta v \in \mathbb{P}_0(K)$ for both $n = 2, 3$.
    \end{itemize}
	We now conclude that $v$ is the solution to a Poisson problem with boundary data given by the degrees of freedom, which provides the result.
\end{proof}

Finally, the global VE spaces are defined on the mesh $\mesh$ as 
\begin{align} \label{eq: global V_h}
    V_h^k&:=\{w_h\in V^k: w_h\lvert_{K}\in V_h^k(K), 
    \ \forall K\in\mesh\}.
\end{align}
Note that, by construction, the discrete spaces are conforming, i.e. $V_h^k \subseteq V^k$ for all $k$.

\subsubsection{Useful inequalities for virtual element spaces}

For ease of reference, we mention several inequalities that were derived for the VE spaces. 

\begin{lemma}[{Nodal inverse inequalities \cite[Thm. 3.6]{Huang2018Errors2D}, \cite{Huang2023Estimates}}]\label{lem:VEM_inverse}
	Let $n=2,3$ and consider $ K\in\mesh$. For each $w\in V^0_h(K)$, it holds that
	\begin{subequations}
	\begin{align}
		%|w|_{1,F}&\lesssim h^{-1}\|w\|_F, & w&\in V^0_h(F), F\in\mcF, \label{eq:VEM_inverse2d} \\ 
		|w|_{1,K} &\lesssim h^{-1}\|w\|_K. \label{eq:VEM_inverse3d}
	\end{align}
	\end{subequations}
\end{lemma}

We also have the nodal VEM trace inequalities which follow from the inverse inequalities.
\begin{lemma}[Nodal trace inequalities]\label{lem:VEM_trace}
	Let $w\in V^0_h(K)$. It holds that 
	\begin{subequations}
	\begin{align}
		\|w\|_E &\lesssim  h^{-1/2}\|w\|_F, & F\in\mcF, E\in\mcE, \quad n&=3, \label{eq:VEM_trace2d} \\ 
		\|w\|_F &\lesssim  h^{-1/2}\|w\|_K, & K\in\mesh, F\in\mcF, \quad n&=2,3. \label{eq:VEM_trace3d}
	\end{align}
	\end{subequations}
\end{lemma}
\begin{proof}
	Note that $w|_F\in V_h^0(F)$. For each case we get then the result in two steps:
	\begin{subequations}
		\begin{align}
			\|w\|_E &\lesssim h^{1/2}|w|_{1,F} + h^{-1/2}\|w\|_F \lesssim h^{-1/2}\|w\|_F, & F\in\mcF, E\in\mcE,\\ 
			\|w\|_F &\lesssim h^{1/2}|w|_{1,K} + h^{-1/2}\|w\|_K \lesssim h^{-1/2}\|w\|_K, & K\in\mesh, F\in\mcF.
		\end{align}
	\end{subequations}
	The first inequality is the continuous scaled trace inequality valid for functions in $H^1(F)$ respectively $H^1(K)$, see \cite[(2.4)]{brenner2017some} and \cite[Lemma 2.1]{DaVeiga2022LowestInterpop}. The second is due to \Cref{lem:VEM_inverse}.
\end{proof}

We recall some more useful results.
\begin{lemma}[{VEM stabilization operators \cite[Prop. 5.2, Thm. A.2, Prop. 5.5]{DaVeiga2022LowestInterpop}}]\label{lem:VEM_stability_ops}
	For $K\in\mesh$ it holds that 
	\begin{align}\label{eq:stability_face}
		\|v_h\|^2_{K} &\eqsim h\sum_{F\in\mcF(K)} (v_h\cdot\nv,v_h\cdot\nv)_F, 
		& v&\in V^{n-1}_h, \\
		\|v_h\|^2_{K} &\eqsim h^2\sum_{E\in\mcE(K)} (v_h\cdot\tv,v_h\cdot\tv)_E, 
		& v&\in V^{n-2}_h, n=3.\label{eq:stability_edge}
	\end{align}
\end{lemma}

% \subsection{Spaces with higher regularity}
\subsubsection{The auxiliary nodal space} \label{sub: The auxiliary nodal space}

Recall the subspace $W^k$ containing $H^1$-regular functions from \eqref{eq:Wk}. To facilitate a discrete version of the regular decomposition from \Cref{lem:cont_decomp}, we introduce a nodal space that captures $H^1$-regular (vector) functions: 
\begin{align}
    W_h^k&:= \left\{w_h\in W^k: w_h\lvert_{K}\in [V^0_h(K)]^{\binom{n}{k}}, 
    \ \forall K\in\mesh \right\}.
\end{align} 
Note that this space contains vector-valued functions for $0<k<n$, and otherwise scalar-valued functions.
We emphasize that these spaces have global $H^1$-regularity; $W_h^k\subset W^k$ for all $k$. Moreover, the inverse (\Cref{lem:VEM_inverse}) and trace (\Cref{lem:VEM_trace}) estimates hold for $0<k<n$ by applying them component-wise.

\subsubsection{A space of regular potentials} \label{sub: A space of regular potentials}
We conclude this section with another convenient subspace of $W^k$, consisting of $H^1$-regular potentials of discrete functions:
\begin{align}
    X^k := \left\{ w\in W^k : dw\in V^{k+1}_h \right\}.
\end{align}
\begin{remark}
	We note that $X^k$ is nontrivial and we will construct specific elements of $X^k$ in \Cref{thm:discrete exact complex,thm:discrete_decomp}.
	% In particular, for given $v_h\in V^{k+1}_h \subseteq V^k$ with $dv_h = 0$, exactness of $(V^\bullet, d)$ ensures existence of $u \in V^{k - 1}$ with $du = v_h$. \Cref{lem:cont_decomp} then allows us to decompose $v=\psi+dp\in V^k$ with $\psi \in W^k$ and $d\psi = d(v - dp) = v_h$. It follows that $\psi\in X^k$.
\end{remark}

For ease of reference, we summarize the function spaces introduced in this section in \Cref{tab:spaces_summary}.

\begin{table}[htb]
	\caption{Instances of the spaces $V_h^k$, $W_h^k$, and $X^k$, for the complex with natural boundary conditions \eqref{eqs:natural_derhamcomplex} and $n=3$.}
	\label{tab:spaces_summary}
	\begingroup
	\renewcommand{\arraystretch}{1.1}
	\begin{tabular}{c|cccc}
		% \hline
		$k$ & $V^k$& $V_{h\mathstrut}^k$ & $W_h^k$ & $X^k$ \\
		\hline
		$0$ & $H^1 / \RR$ & Nodal space  & Nodal space & Potentials $v\in H^1$ with $\grad v \in V^{n-2}_h$ \\
		$n-2$ & $H^{\crl}$ &Edge space  & Vector nodal space  & Potentials $v\in [H^1]^3$ with $\crl v \in V^{n-1}_h$ \\
		$n-1$ & $H^{\dive}$ &Face space  & Vector nodal space  & Potentials $v\in [H^1]^3$ with $\dive v \in V^n_h$ \\
		$n$ & $L^2$ &Piecewise constants & Nodal space & Potentials $v\in H^1$ \\
	\end{tabular}
	\endgroup
\end{table}

\section{Regular decomposition for virtual elements}\label{sec: VEM Decomposition}

%In this section we introduce the mesh and the discrete spaces of interest. 
%We want to perform auxiliary space preconditioning on the resulting system matrix of a virtual element method, so we must introduce the virtual element nodal, edge, and facet spaces. 

In this section we develop the notions of co-chain complex and regular decomposition in the discrete setting, which are essential for our construction of nodal auxiliary space preconditioners. % We start by introducing the concepts at the continuous level. 
We show that the virtual element method forms a discrete co-chain complex, and that the canonical interpolation operators are stable co-chain projections. With the help of a Clément interpolant, we then derive the stable discrete regular decomposition.

\subsection{A discrete co-chain complex}\label{sec:cochain}

We continue by placing the discrete virtual element spaces from \Cref{sec:vemintro}, illustrated in \Cref{fig:DeRhamDodecahedron}, in the context of co-chain complexes from \Cref{sec:cochaincomplexes}. Several results presented in this section have been shown for three-dimensional virtual element spaces in \cite{DaVeiga2018Lowest,DaVeiga2022Maxwell}. While the two-dimensional analogues can be derived as boundary cases, we include the proofs here for completeness.

\begin{figure}[htb]
    \begin{centering}
    \includegraphics[scale=0.23]{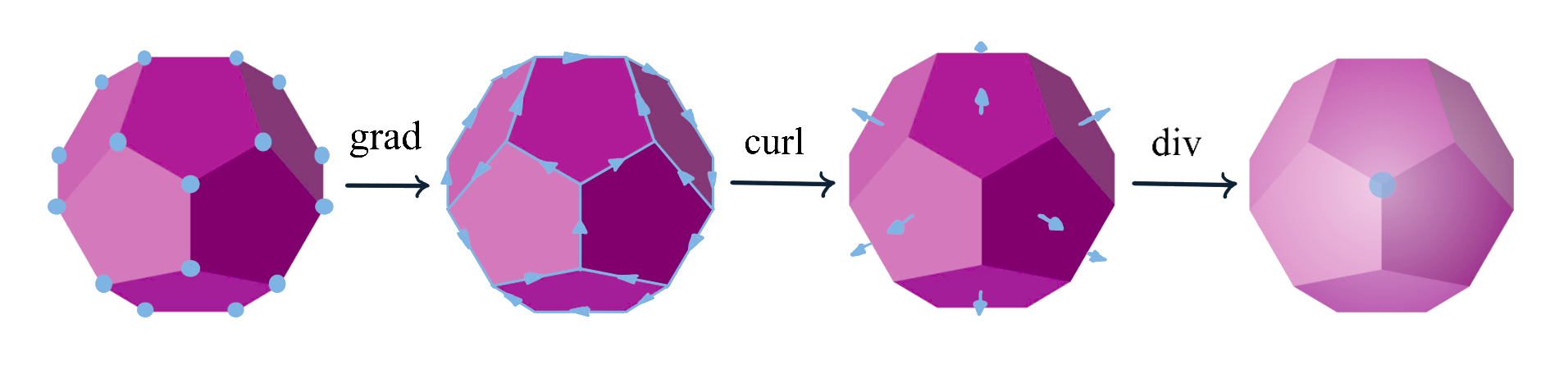}
    \caption{A discrete co-chain complex formed by the nodal, edge, facet, and cell virtual element spaces $V_h^k$ on a regular dodecahedron.}
    \label{fig:DeRhamDodecahedron}
    \end{centering}
\end{figure}

\begin{lemma}\label{lem:discrete complex}
    For $V_h^k$ defined in \Cref{sec:vemintro}, let $d$ be the differential corresponding to Sobolev space $V^k$ from \Cref{sec:cochaincomplexes}. Then $(V^\bullet_h, d)$ is a co-chain complex:
    \begin{align}
        dV_h^k \subseteq \ker_d(V^{k+1}_h).
    \end{align}
\end{lemma}
\begin{proof}
    The result was shown for $n = 3$ in \cite{DaVeiga2018Lowest,DaVeiga2022Maxwell}, so we continue with $n = 2$.
	Since $dd = 0$, it suffices to show that $dV_h^k \subseteq V^{k+1}_h$. 
    Let us consider the two non-trivial cases for an element $K\in\mesh$:
    \begin{itemize}
        \item ($k=0$)
        Let $v\in V^0_h(K)$ and $w := \crl v$. We confirm that $w \in V_h^{n-1}(K)$ according to definition \eqref{eqs: def edge space}. First, we immediately have $\dive w=0$. Second, $\crl w = \Delta v \in \mathbb{P}_0(K)$ by \eqref{eqs: def nodal space}.
        Third, on a facet $F$, we have $w\cdot\nv = \nabla v \cdot \tv \in \mathbb{P}_0(F)$ since $v|_F\in\mathbb{P}_1(F)$ by definition of $V^0_h(K)$. 
		Lastly 
        \begin{align}
            \int_K w \cdot \xv_K^{\perp} =\int_K \grad v \cdot \xv_K = 0.
        \end{align}
        \item ($k=1$)
        Let $v\in V^{n-1}_h(K)$, then $\dive v \in \mathbb{P}_0(K) = V_h^n(K)$ by \eqref{eqs: def facet space}.
    \end{itemize}
	Hence $d V_h^k(K) \subseteq V_h^{k + 1}(K)$ locally on each $K \in \mesh$. Moreover, $dd = 0$ implies that $dV_h^k \subseteq V^{k + 1}$ globally and the result follows by the definition \eqref{eq: global V_h}.
\end{proof}

\subsection{Stable co-chain projections} \label{sec:commuting_ops}

Each $V_h^k$ is endowed with a canonical interpolation operator $\Pi^k_h:[C^{\infty}(\Omega)]^{\binom{n}{k}} \to V_h^k$, which is defined as follows using the degrees of freedom $\text{dof}_j^k$ and basis functions $b_j$:
\begin{align}
	[C^{\infty}(\Omega)]^{\binom{n}{k}}\ni v \mapsto \sum_{j=1}^{\dim V_h^k} \text{dof}^k_j(v) b_j\in  V_h^k. \label{eq:interpop}
\end{align}
Note that each $\Pi^k_h$ corresponds to an $L^2$-projection onto a $k$-dimensional mesh entity.
We refer to an operator $\Pi_h^k$ as a \emph{co-chain projection} if it commutes with the differential operator \cite{ArnFEEC18}. This commuting property is visualized by the following diagram:
%
%https://tikzcd.yichuanshen.de/#N4Igdg9gJgpgziAXAbVABwnAlgFyxMJZABgBpiBdUkANwEMAbAVxiRGJAF9T1Nd9CKAIzkqtRizYBBLjxAZseAkQBMo6vWatEIGd16KBRAMzrxW6bIP9lKACxnNknXrkKbg5AFZHE7eyt5PiVPESExJ39XaxDVUnCNP0t9IMNbZFME82ddQPdY+3iIpJc84KMUMizItg4U-IrvIsSLHTqxGCgAc3giUAAzACcIAFskMhAcCCQhFKHRmeoppDVs-wAddbA6ACMGOkD5scRV5cRTNbZNnBgADxxgAGMmQYZOQ+Hji7OHS51ru4PKBYGjvOafcZLaaIESTOhYBhsAAWEAgAGsPgtEL8zj4QEiYHQoGwcAB3CAEokIcFY2FnABs1EpxJ0ZIphKg1LkRxWUKQAHYmRySeTmVyBhDznzEAAOIVEkXsqmY4446EATnlLMmoo54pAPJhEwZS3hiJ0KPRKqQjMm0MFfxAm22ewONOODrOcsdAPuTxeb2tsulmp96xufuBoKDobOQlhzMVYq4FE4QA
% \begin{figure}[htb]
% \centering
\begin{subequations}\label{diagram:commuting3D}
\begin{align}
\begin{tikzcd}[
	ampersand replacement=\&,
  ]
0 \arrow[r] \arrow[r, hook] \& C^{\infty}(\Omega) / \RR \arrow[r, "\grad"] \arrow[d, "\Pi^{0}_h"] \& C^{\infty}(\Omega)^3 \arrow[r, "\crl"] \arrow[d, "\Pi^{n-2}_h"] \& C^{\infty}(\Omega)^3 \arrow[r, "\text{div}"] \arrow[d, "\Pi^{n-1}_h"] \& C^{\infty}(\Omega) \arrow[r] \arrow[d, "\Pi^n_h"] \& 0 \\
0 \arrow[r, hook]           \& V_h^{0} \arrow[r, "\grad"]                      \& V_h^{n-2} \arrow[r, "\crl"]                      \& V_h^{n-1} \arrow[r, "\text{div}"]                      \& V_h^0 \arrow[r]                      \& 0
\end{tikzcd} 
&& n &= 3
\\
% \newline\newline\newline
\begin{tikzcd}[
	ampersand replacement=\&,
  ]
    0 \arrow[r] \arrow[r, hook] \& C^{\infty}(\Omega) / \RR \arrow[r, "\crl"] \arrow[d, "\Pi^{0}_h"] \& C^{\infty}(\Omega)^2 \arrow[r, "\dive"] \arrow[d, "\Pi^{n-1}_h"] \& C^{\infty}(\Omega) \arrow[r] \arrow[d, "\Pi^{n}_h"] \& 0 \\
    0 \arrow[r, hook]           \& V_h^{0} \arrow[r, "\crl"]    \& V_h^{n-1} \arrow[r, "\dive"]    \& V_h^n \arrow[r]    \& 0
\end{tikzcd}
&& n &= 2
% \caption{Commuting diagram of the VEM, for $n=3$ and $n=2$.}
% \label{diagram:commuting3D}
% \end{figure}
\end{align}
\end{subequations}

% \begin{figure}[H]
% \begin{centering}
% \begin{tikzcd}
% 0 \arrow[r] \arrow[r, hook] & (C^{\infty}(\Omega))^{n\choose n-3} \arrow[r, "d"] \arrow[d, "\Pi^{n-3}_h"] & (C^{\infty}(\Omega))^{n\choose n-2} \arrow[r, "d"] \arrow[d, "\Pi^{n-2}_h"] & (C^{\infty}(\Omega))^{n\choose n-1} \arrow[r, "d"] \arrow[d, "\Pi^{n-1}_h"] & C^{\infty}(\Omega) \arrow[r] \arrow[d, "\Pi^{n}_h"] & 0 \\
% 0 \arrow[r, hook]           & V_h^{n-3} \arrow[r, "d"]                      & V_h^{n-2} \arrow[r, "d"]                      & V_h^{n-1} \arrow[r, "d"]                      & V_h^n \arrow[r]                      & 0
% \end{tikzcd}
% \caption{Commuting diagram of the VEM.}
% \label{diagram:commuting}
% \end{centering}
% \end{figure}

However, the space $C^\infty$ is too restrictive for our purposes. We therefore extend the domain of $\Pi_h^k$ to the Hilbert spaces introduced in \Cref{sub: The auxiliary nodal space,sub: A space of regular potentials}.

\begin{lemma} \label{lem:domain Pi}
    The operator $\Pi_h^k$ is well-defined on $X^k \cup W_h^k$.
\end{lemma}
\begin{proof}
    First, the space $W_h^k$ consists of nodal VE spaces which are bounded, continuous functions. The evaluation of the degrees of freedom is therefore well-defined for each $k$. 

    For $X^k$, we consider five cases and it is sufficient to consider the local spaces on an element $K\in\mesh$.
	\begin{itemize}
	\item ($k=0$, $n=2$) 
        By \Cref{lem:subspaces_H1}, $\crl v\in V^{n-1}_h$ implies $(\crl v)|_K  \in [H^1(K)]^2$ whereby $v\in H^2(K).$ 
        In turn, $v|_{\partial K} \in H^{3/2}(\partial K)$ which implies that the restriction of $v$ on the skeleton $\bigcup_{F\in\mcF}F$ is equal to a continuous function in the sense of $L^2$. 
        This continuous representative allows us to define a nodal evaluation of $v$.
	\item ($k=0$, $n=3$)
		By \Cref{lem:subspaces_H1}, $(\grad v)|_K \in V^{n-2}_h\subset [H^1(K)]^3$ so that $v|_K \in H^2(K)$. Applying consecutive trace inequalities gives us 
			$(v|_F)|_{\partial F} \in H^{2-1/2-1/2}(\partial F)=H^1(\partial F)$. In turn, the nodal values of $v$ on the mesh-skeleton are well-defined.
	\item ($k=n-2$, $n=3$)
		For $v \in X^{n-2}$, we have $\crl v \in V^{n-1}_h$ and, in turn, \Cref{lem:subspaces_H1} gives us $\crl v\in [H^1(K)]^3$. Consequently $\int_E v\cdot \tv$ is well-defined, c.f. \cite[Prop. 4.5]{DaVeiga2022LowestInterpop}.
	\item ($k=n-1$)
		Since $X^{n-1} \subset W^{n-1}$, each $v \in X^{n-1}$ has a well-defined normal trace $(\nv \cdot v)|_F \in H^{\frac12}(F) \subset L^2(F)$ on the mesh facets $F$.
	\item ($k=n$)
		In this case, $\Pi^n_h$ is the $L^2$-projection onto the elementwise constants, which is well-defined on $L^2(\Omega)\supset X^n$. 
	\end{itemize}
	The evaluation of the degrees of freedom is thus well-defined for both $X^k$ and $W_h^k$, and so is the interpolant.
\end{proof}

With the domain of $\Pi^{\bullet}_h$ extended from $C^\infty$, we now prove its most important property in the following lemma.

\begin{lemma}\label{lem:co-chain projection}
	$\Pi_h^\bullet$ is a co-chain projection, i.e. 
    \begin{align}
        \Pi_h^{k+1} d = d \Pi_h^k.
    \end{align}
\end{lemma}
\begin{proof}
	Again, the $n=3$ case was shown in \cite{DaVeiga2022Maxwell} and we focus on $n = 2$. Fix $K\in\mesh$ and let $v$ be a sufficiently regular function to be in the domain of $\Pi_h^k$. We divide into cases of $k=0$ and $k=1.$ 
	% We first show that the degrees of freedom of $V^{n-1}_h$ (resp. $V^{n}_h$) coincide, and then we show that the functions are actually in $V^{n-1}_h$ (resp. $V_h^n$). 
	By \Cref{lem:domain Pi} it is enough to show that the degrees of freedom of $V^{n-1}_h$ respectively $V^{n}_h$ coincide.

	\begin{itemize}
		\item ($k=0$) 
		%Let $v\in H^{2}(K)$.
		Note that for $F\in \mcF(K)$
		\begin{align}
			\dof^{n-1}_F(\Pi_h^{n-1} (\crl v)) = \dof^{n-1}_F(\crl v) = \int_F \nv\cdot(\crl v) = \int_F \tv\cdot \grad v &= v(F_1)-v(F_0) ,\\
			\dof^{n-1}_F(\crl(\Pi^0_hv)) = \int_F \nv\cdot(\crl \Pi^0_hv) = \int_F \tv\cdot \grad (\Pi^0_hv) = \Pi_h^0v(F_1)-\Pi_h^0v(F_0) &= v(F_1)-v(F_0).
		\end{align}
		% We have $\Pi^0_hv\in V^0_h(K)$ and we wonder if $w:=\crl\Pi^0_h v\in L^2(K)$ is in fact in $V^{n-1}_h(K)$. Clearly $\dive w=0$ and $\crl w\in \mathbb{P}_0$ since $\Delta \Pi^0_h v\in \mathbb{P}_0.$ It is also clear that $(w\cdot\nv)|_F\in\mathbb{P}_0$ since $\Pi^0_h v|_F\in\mathbb{P}_1$. Lastly 
		% \begin{align}
		% 	\int_K w \cdot \xv_K^{\perp} =\int_K \grad (\Pi^0_h v) \cdot \xv_K = 0.
		% \end{align}
		\item ($k=1$)
		%Let $v\in [H^1(K)]^2$.
		We have
		\begin{align}
			\dof^{n}_K(\Pi^{n}_h(\dive v)) &= \dof^{n}_K(\dive v) = \int_K \dive v ,\\
			\dof^{n}_K(\dive(\Pi_h^{n-1} v)) &= \int_K \dive(\Pi_h^{n-1} v) = \int_{\partial K} \Pi_h^{n-1}  v \cdot \nv = \int_{\partial K} v \cdot \nv = \int_K \dive v.
		\end{align}
		% With $\Pi^{n-1}_hv\in V^{n-1}_h(K)$ we now wonder whether $\dive \Pi^{n-1}_hv \in \mathbb{P}_0(K).$ This is of course true by definition of the $n=2$ facet space. \wmb{?? shorter?}
	\end{itemize}
	
\end{proof}

\subsubsection{Exactness of the discrete complex}

The combination of the co-chain projection $\Pi_h^k$ with the regular decomposition of $V^k$ from \Cref{sub: Regular decomposition} now allow us to prove that $(V_h^\bullet, d)$ is exact. While this result was shown in \cite{DaVeiga2018Lowest} for $n = 3$, we present it here as an alternative proof.

\begin{theorem}\label{thm:discrete exact complex}
    $(V_h^\bullet, d)$ is an \emph{exact} co-chain complex, i.e.
    \begin{align}
        dV_h^k = \ker_d(V^{k+1}_h).
    \end{align}
\end{theorem}
\begin{proof}
	The inclusion $dV_h^k \subseteq \ker_d(V^{k+1}_h)$ was shown in \Cref{lem:discrete complex} so we proceed to prove the converse inclusion.
	Let $v_h \in V_h^{k + 1}$ with $dv_h = 0$. Since $V_h^{k + 1} \subset V^k$, the exactness of $(V^\bullet, d)$ implies that a $u \in V^k$ exists with $du = v_h$. The regular decomposition from \Cref{lem:cont_decomp} then guarantees that $\psi \in W^k$ and $p \in V^{k - 1}$ exist such that $u = \psi + dp$. In turn, we derive that $d\psi = d(u - dp) = v_h$ and thus $\psi \in X^k$. Let us now set $u_h := \Pi_h^k \psi\in V^k_h$, which is well-defined by \Cref{lem:domain Pi}. Using the commuting property of \Cref{lem:co-chain projection}, we obtain
	\begin{align}
		d u_h = d \Pi_h^k \psi = \Pi_h^{k + 1} d \psi = \Pi_h^{k + 1} v_h = v_h.
	\end{align}
	Thus, we have shown that $\ker_d(V^{k+1}_h) \subseteq dV_h^k$.
\end{proof}

\begin{corollary}[Stable discrete potential] \label{cor: stable_pot_h}
    Given $v_h\in V^{k}_h$ with $dv_h=0$, there exists a discrete potential $u_h\in V^{k-1}_h$ such that
    \begin{align}
        du_h &= v_h, &
        \|u_h\|_{V^{k - 1}} &\lesssim \|v_h\|.
    \end{align}
\end{corollary}
\begin{proof}
    Due to \Cref{thm:discrete exact complex}, 
	the result follows directly from \Cref{lem:stable_pot}.
\end{proof}

\subsubsection{Approximation properties for regular functions}

In addition to the commuting property, the interpolation operators have optimal approximation properties when applied to sufficiently regular functions. We summarize these as a general result in the following theorem, and outline the details for the edge space in a separate, subsequent lemma. 

\begin{theorem}[Approximation properties]\label{thm:pot_interpol}
    For $v\in X^k \cup W_h^k$, we have 
    \begin{align}
    	\| (I-\Pi^k_h) v\| \lesssim h\|v\|_1.\label{eq:interp_est}
    \end{align}
\end{theorem}
\begin{proof}
    Let us consider four cases:
    \begin{itemize}
		\item ($k=0$) 
		If $v\in W^{0}_h=V^0_h$ then $v=\Pi^0_h v$ and the result is trivial.
		For $v\in X^0$, we use \Cref{lem:co-chain projection} to observe that $d \Pi^0 v = \Pi^1 d v = dv$. In turn, $(I - \Pi^0)v \in \ker_d(V^0) = \{ 0 \}$ by the exactness of $(V^\bullet, d)$ and we again conclude $v = \Pi^0 v$. 
	    \item ($k=n-2$, $n=3$)
	    The result follows from \Cref{lem:edge_interpol}, proven below, after summing over all elements $K$.
		\item ($k=n-1$)
    	Since $X^{n-1}(K) \cup W_h^{n-1}(K) \subset W^{n-1}(K)$, the result follows from \cite[Prop. 3.2]{DaVeiga2022LowestInterpop} combined with the convexity assumption \eqref{enum:ass_4} from \Cref{ass:mesh}.
		\item ($k=n$)
	    In this case, $\Pi^n_h$ is the $L^2$-projection onto the elementwise constants. The difference $(I-\Pi^n_h)v$ therefore has zero mean on each element $K$. The result follows from the Poincaré inequality after summing over all $K \in \mesh$.
    \end{itemize}
\end{proof}

\begin{lemma}[Edge interpolant]\label{lem:edge_interpol}
	For $n=3$ and $K\in\mesh$, we have
	\begin{align}
		\|(I-\Pi^{n-2}_h) v\|_K &\lesssim 
		h\|v\|_{1,K},& \forall v&\in X^{n-2}\cup W^{n-2}_h  \label{eq:edge_continuity_convex}
	\end{align}
\end{lemma}
\begin{proof}
	% We first note that, regardless of if $v\in X^{n-2}(K)$ or $v\in W^{n-2}_h(K)$, we have $v\in W^{n-2}(K)$.
	% If $v\in W^{n-2}_h(K)$, each component $v_i\in V^0_h(K)$ and so it holds that $\grad v_i\in V^{n-2}_h$ by \Cref{lem:discrete complex}, and thus \Cref{lem:subspaces_H1} gives us $\grad v\in [H^1(K)]^{3\times 3}$ whereby $\crl v\in [H^{1}(K)]^3$.
	% On the other hand if $v\in X^{n-2}(K)$, we have $\crl v \in V^{n-1}_h$ and, in turn, \Cref{lem:subspaces_H1} gives us $\crl v\in W^{n-1}(K)\subset [H^1(K)]^3$.
	% From \cite[Prop. 4.5]{DaVeiga2022LowestInterpop} we therefore have for all $1/2<s\le 1$
	% \begin{align}
	% 	\|(I-\Pi^{n-2}_h) v\|_K &\lesssim h^s|v|_{s,K}+h\|\crl v\|_K+h^{s+1}|\crl v|_{s,K},& \text{for }v\in X^{n-2}\cup W^{n-2}_h.
	% 	%h\|v\|_{1,K}+h^2\|\grad\crl v\|_K, 
	% 	%\label{eq:edge_continuity}
	% \end{align}
	% Taking $s=1$ and bounding $\|\crl v\|_K\le \| v \|_{1, K}$, we get 
	% \begin{align}
	% 	\|(I-\Pi^{n-2}_h) v\|_K &\lesssim %h^s|v|_{s,K}+h\|\crl v\|_K+h^{s+1}|\crl v|_{s,K}, 
	% 	h\|v\|_{1,K}+h^2|\crl v|_{1,K}.
	% 	\label{eq:edge_continuity}
	% \end{align}
	Let us divide into two cases depending on if $v\in X^{n-2}(K)$ or $v\in W^{n-2}_h(K)$.
	\begin{itemize}
		\item ($v\in X^{n-2}(K)$)
		By definition of the $X^k$ spaces, we have $v\in W^{n-2}(K)$ and $\crl v \in V^{n-1}_h$. In turn, \Cref{lem:subspaces_H1} gives us $\crl v\in W^{n-1}(K)$.
		From \cite[Prop. 4.5]{DaVeiga2022LowestInterpop} we therefore have for all $1/2<s\le 1$
		\begin{align}
			\|(I-\Pi^{n-2}_h) v\|_K &\lesssim h^s|v|_{s,K}+h\|\crl v\|_K+h^{s+1}|\crl v|_{s,K}.
		\end{align}
		Taking $s=1$ and bounding $\|\crl v\|_K\le \| v \|_{1, K}$, we get 
		\begin{align}
			\|(I-\Pi^{n-2}_h) v\|_K &\lesssim 
			h\|v\|_{1,K}+h^2|\crl v|_{1,K}.
			\label{eq:edge_continuity}
		\end{align}
		Let $w_h:=\crl v\in V^{n-1}_h(K)\cap W^{n-1}(K)$. We can invoke \cite[Lem. 3]{DaVeiga2022Maxwell} on $w_h$
		\begin{align}
			|w_h|_{s',K} \lesssim h^{-s'}\|w_h\|_K,
		\end{align}
		for some $1/2<s'\le 1$. This $s'$ is the minimum value for which both of the following statements hold:
		\begin{itemize}
			\item The solution to the Laplace problem with Neumann boundary data $w_h\cdot \vec{\nu}\in H^{1/2}(\partial K)$ on $K$ is in $H^{s'+1}(K)$.
			\item The spaces $H^{\dive}(K)$ and $H^{\crl}(K)$ are continuously embedded in $H^{s'}(K)$.
		\end{itemize}
		By \eqref{enum:ass_4} of \Cref{ass:mesh}, both statement hold for $s'=1$, see \cite[Thm. 2, Thm. 4, page 16]{DaugePolyhedralRegularity}\cite[Cor. 23.5]{Dauge2006Elliptic} respectively \cite[Prop. 3.7, Cor. 23.5]{Amrouche1998Vector}. Therefore, we can bound the second term:
		\begin{align}
				\|(I-\Pi^{n-2}_h) v\|_K &\lesssim 
				h\|v\|_{1,K}+h^2 |\crl v |_{1, K} 
				\lesssim h\|v\|_{1,K}+h \|\crl v \|_K 
				\lesssim h\|v\|_{1,K}.
		\end{align}
		\item ($v\in W^{n-2}_h(K)$) 
		Let $\bar{v}\in V^{n-2}_h(K)$ be the mean of $v$ over $K$. Then by the Poincaré inequality
		\begin{align} \label{eq: W estimate 1}
			\|(I-\Pi^{n-2}_h) v\|_K &\leq \|v-\bar{v}\|_K + \|\bar{v}-\Pi^{n-2}_h v\|_K \lesssim h|v|_{1,K}+\|\bar{v}-\Pi^{n-2}_h v\|_K.
		\end{align}
		We estimate the second term. 
		The interpolant preserves constants, so $\Pi^{n-2}_h\bar{v}=\bar{v}$. 
		% Consider $\Pi^{n-2}_h(\bar{v}-v)=\sum_{E\in\mcE}(\int_E(\bar{v}-v)\cdot \tv) \phi_E$, $\phi_E\in V^{n-2}_h$ such that $\int_E\phi_E\cdot\tv = 1$. 
		% Lastly we recall the nodal VEM trace inequality 
		% \begin{subequations}\label{eq:VEM_trace}
		% \begin{align}
		% 	\|w\|_E \lesssim h^{1/2}|w|_{1,F} + h^{-1/2}\|w\|_F \lesssim h^{-1/2}\|w\|_F, & w&\in V^0_h(F), F\in\mcF, E\in\mcE \label{eq:VEM_trace2d} \\ 
		% 	\|w\|_F \lesssim h^{1/2}|w|_{1,K} + h^{-1/2}\|w\|_K \lesssim h^{-1/2}\|w\|_K, & w&\in V^0_h(K), F\in\mcF, K\in\mesh. \label{eq:VEM_trace3d}
		% \end{align}
		% \end{subequations}
		% The first inequality is the continuous scaled trace inequality valid for functions in $H^1(F)$ respectively $H^1(K)$, see e.g. \cite{Cangiani2017Posteriori}. The second inequality is the VEM inverse inequality, see \cite[Thm. 3.6]{Huang2018Errors2D} for \eqref{eq:VEM_trace2d} and \cite{Huang2023Estimates} for \eqref{eq:VEM_trace3d}, applied to the $H^1-$seminorm. 
		We now use \cite[Cor. 4.3]{DaVeiga2022LowestInterpop}, the definition of the edge interpolant as an $L^2$-projection on edges, %Hölder inequality, 
		VEM trace inequalities in 2D \eqref{eq:VEM_trace2d} and 3D \eqref{eq:VEM_trace3d}, and a Poincaré inequality:
		\begin{align}
			\|\bar{v}-\Pi^{n-2}_h v\|_K = \|\Pi^{n-2}_h (\bar{v}-v)\|_K &\lesssim h\sum_{F\in\mcF(K)}\sum_{E\in\mcE(F)} \|\Pi^{n-2}_h(\bar{v}- v) \cdot \tv\|_E \nonumber\\
			%&= h\sum_{F\in\mcF(K)}\sum_{E\in\mcE(F)} \|(\int_E (\bar{v}- v)\cdot\tv)\phi_E\cdot\tv\|_E \\
			%&= h\sum_{F\in\mcF(K)}\sum_{E\in\mcE(F)} |\int_E (\bar{v}- v)\cdot\tv|\|1/|E|\|_E \\
			%&= h\sum_{F\in\mcF(K)}\sum_{E\in\mcE(F)} 1/|E|^{1/2}|\int_E (\bar{v}- v)\cdot\tv|\\
			&\leq h\sum_{F\in\mcF(K)}\sum_{E\in\mcE(F)}\|(\bar{v}- v)\cdot\tv\|_E \nonumber \\
			&\lesssim h^{1/2} \sum_{F\in\mcF(K)} \|\bar{v}- v\|_F %\nonumber \\
			% &\lesssim h^{1/2} (h^{1/2}|\bar{v}-v|_{1,K}+h^{-1/2}\|\bar{v}-v\|_K) \\
			% &= h |v|_{1,K} + \|\bar{v}-v\|_K \\
			\lesssim \|\bar{v}-v\|_K
			\lesssim  h |v|_{1,K}. \label{eq: W estimate 2}
		\end{align}
		% Let $w:=\crl v$. By commutativity of partial derivatives we have $\Delta w=0$ since $\Delta v =0$. Moreover, since $(\crl v)|_F = \crl (v|_F)$,... 
		% \begin{align}
		% 	|w|^2_{1,K} &= -\int_K \Delta w \cdot w + \int_{\partial K} (\grad w \cdot n) \cdot w = \int_{\partial K} (\grad w \cdot n) \cdot w \\
		% 	&\leq \|\grad w\cdot n\|_{H^{-1/2}(\partial K)}\|w\|_{H^{1/2}(\partial K)} \leq (\|\grad w\|_K + h\|\Delta w\|_K)\|w\|_{H^{1/2}(\partial K)} \\
		% 	&=|w|_{1,K}\|w\|_{H^{1/2}(\partial K)} \leq |w|_{1,K}h^{-1/2}\|w\|_{\partial K} 
		% \end{align}
		The combination of \eqref{eq: W estimate 1} and \eqref{eq: W estimate 2} yields the result.
	\end{itemize}
\end{proof}

% \begin{corollary}[Interpolation estimates for $W_h^k$]\label{cor:vecnode_interpol}
%     For $v\in W_h^k$, we have the same estimate
%     \begin{align}
%     	\| (I-\Pi^k_h) v\| \lesssim h\|v\|_1.\label{eq:interp_est_W}
%     \end{align}
% \end{corollary}
% \begin{proof}
% 	Note in the proof of \Cref{thm:pot_interpol} how every regularity argument using \Cref{lem:subspaces_H1} was needed only to be able to define the interpolant. For $v\in W_h^k$ we can always define $\Pi^k_hv$ regardless of $k$. The remaining arguments follow word for word.
% \end{proof}

\subsubsection{Stability of the co-chain projection}

A direct consequence of the approximation property \Cref{thm:pot_interpol}, is that the interpolation operators are stable.

\begin{corollary}[Stability in the $L^2$-norm]\label{cor:pot_stability}
    For $v\in X^k \cup W_h^k$, we have 
    \begin{align}
    	\| \Pi^k_h v\| &\lesssim \|v\|_1.
    \end{align}
\end{corollary}
\begin{proof}
	An application of the triangle inequality with \Cref{thm:pot_interpol} leads us to
	\begin{align*}
		\| \Pi^k_h v\| &\leq \| \Pi^k_h v - v\| + \|v\| \lesssim h\|v\|_1 + \|v\|_1 \lesssim \|v\|_1.
	\end{align*}
\end{proof}

\begin{remark} \label{rem: Pi^n is L2}
	For $k = n$, we recall that $\Pi_h^n$ is an $L^2$-projection onto the piecewise constants, which is even stable in $L^2$, i.e. $\| \Pi_h^n v \| \lesssim \| v \|$ for all $v \in V^n$.
\end{remark}

However, we require a slightly stronger result in our analysis. Let us therefore focus on the auxiliary nodal spaces $W_h^k$ and consider stability of $\Pi_h^k$ in the $V^k$-norm.
\begin{theorem}[Stability in the $V^k$-norm]\label{cor:Wk_stab_in_Vk}
	For $v\in W^k_h$,
	\begin{align}
		\|\Pi^k_h v\|_{V^k} \lesssim \|v\|_1.
	\end{align}
\end{theorem}
\begin{proof}
	%The first inequality \eqref{eq:Pistability_on_Whk} is a kind of stability property of $\Pi^k_h$ applied to functions in $W^k_h$. 
    % Since every interpolation operator $\Pi^k_h,\ k=0,1,\dots,n$, is well defined on $W^k_h$, we can invoke \Cref{cor:pot_stability} to get 
    % \begin{align}
    %     \|\Pi^k_h v\|^2 \lesssim \|v\|^2_1.
    % \end{align}
	By \Cref{cor:pot_stability} it, suffices to prove $\|d\Pi^k_hv\|\lesssim \|v\|_1$. We consider four cases:
    \begin{itemize}
		\item ($k=0$)
        Since $W^0_h = V^0_h$, we have that $\Pi^0_h W^0_h = W^0_h$. Thus, $\| d \Pi^0_h v\| = \| d v\| \le \| v \|_1$ for $v\in W^0_h$.
        % Fix $K\in\mesh$ and set $\tilde v \in W_h^0(K)$ and consider $v = \tilde v - \Pi^n_K \tilde v$ in which $\Pi^n_K$ is the projection onto the constant vectors on $K$. (Note that $\Pi^n_K = \Pi^n_h|_K$.) Then $\Pi^0_h \mathbb{P}_0(K)^n=\mathbb{P}^n_0(K)$ so we can work with $v$ and apply the Poincaré inequality for zero mean vector fields at the end to recover some powers of $h$. 
		% If $n=2$, $\|\grad \cdot\|=\|\crl \cdot\|$, and we use the facet stabilization inequality \eqref{eq:stability_face}, polynomial trace inequality, \eqref{eq:VEM_inverse2d}, and Poincaré:
        % \begin{align}
        %     \|\crl \Pi^0_h \tilde v \|_K &= \|\crl \Pi^0_h v \|_K = \|\Pi^1_h \crl v \|_K \lesssim h\sum_{\mcF(K)} \|\crl v \cdot \nv \|^2_F \\
		% 	&\lesssim \|\grad v \|^2_K \lesssim h^{-2}\| v \|^2_K \lesssim | v |^2_{1,K} = |\tilde v |^2_{1,K}.
        % \end{align}
		% If $n=3$ we use the edge stabilization inequality \eqref{eq:stability_edge}, polynomial trace inequality, \eqref{eq:VEM_inverse2d}, \eqref{eq:VEM_trace3d}, and Poincaré:
        % \begin{align}
        %     \|\grad \Pi^0_h \tilde v \|_K &= \|\grad \Pi^0_h v \|_K = \|\Pi^1_h \grad v \|_K \lesssim h^2\sum_{\mcE(K)} \|\grad v \cdot \tv \|^2_E \\
		% 	&\lesssim h\sum_{\mcF(K)} \|\grad v \|^2_F \lesssim h^{-1}\sum_{\mcF(K)} \| v \|^2_F \lesssim h^{-2}\| v \|^2_K \lesssim | v |^2_{1,K} = |\tilde v |^2_{1,K}.
        % \end{align}
		% Summing over $K\in\mesh$ gets the result. 
        \item ($k=n-2$, $n=3$) 
        Fix $K\in\mesh$ and let $w \in W_h^{n-2}(K)$ and consider $v = w - \bar{w}$, in which $\bar{w}$ is the mean of $w$ on $K$. Since $\Pi_h^{n-2} \mathbb{P}_0(K)^3 = \mathbb{P}_0(K)^3$, we have
        \begin{align}
            \crl \Pi_h^{n-2} w
            = \crl \Pi_h^{n-2} v
            + \crl \Pi_h^{n-2} \Pi^n_K w
            = \crl \Pi_h^{n-2} v.
        \end{align}
        We continue by bounding the norm on the right-hand side.
        The commuting property from \Cref{lem:co-chain projection}  and \cite[Prop. 3.1]{DaVeiga2022LowestInterpop} give us
        \begin{align} \label{eq: bound1 curl}
            \| \crl \Pi_h^{n-2} v \|_K
            &= \| \Pi_h^{n-1} \crl v \|_K 
            \lesssim h^{\frac12} \| \nu \cdot (\Pi_h^{n-1} \crl v) \|_{\partial K}
            = h^{\frac12} \sum_{F \in \mcF(K)} \| \nu \cdot (\Pi_h^{n-1} \crl v) \|_F
        \end{align}
        The operator $\Pi_h^{n-1}$ takes the mean of the normal component on each face. Since $v \in W_h^{n-2}$, $v$ is linear on each edge, and this leads us to
        \begin{align*}
            \nu \cdot (\Pi_h^{n-1} \crl v)|_F
            = \frac1{|F|} \int_F \nu \cdot (\crl v)
            = \frac1{|F|} \int_F \dive_F (\nu \times v)
            = \frac1{|F|} \int_{\partial F} \tv \cdot v
            &= \sum_{e \in \mcE(F)} \frac{|e|}{|F|} \sum_{n \in \mcN(e)} \frac{\tv_e \cdot v(n)}2 \\
            &\lesssim h^{-1} \sum_{n \in \mcN(F)} | v(n) |.
        \end{align*}
        This gives us the following bound on the norm on $F$:
        \begin{align*}
            \| \nu \cdot (\Pi_h^{n-1} \crl v) \|_F
            \lesssim \left( h^{-1} \sum_{n \in \mcN(F)} |v(n)| \right) \| 1 \|_F
            \lesssim \sum_{n \in \mcN(F)} |v(n)|.
        \end{align*}
        Using this bound on \eqref{eq: bound1 curl} and \cite[Thm. 4.2]{Huang2023Estimates}, we obtain the estimate
        \begin{align}
            \| \crl \Pi_h^{n-2} v \|_K
            \lesssim h^{\frac12} \sum_{n \in \mcN(K)} |v(n)|
            \lesssim h^{\frac12} \left(\sum_{n \in \mcN(K)} |v(n)|^2 \right)^{\frac12}
            \lesssim h^{-1} \| v \|_K
        \end{align}
        Finally, we recall that $v$ has zero mean by definition. Hence, the Poincar\'e inequality can be invoked, giving us
        \begin{align}
            \| \crl \Pi_h^{n-2} w \|_K=\| \crl \Pi_h^{n-2} v \|_K
            \lesssim | v |_{1, K} = | w |_{1, K}. \label{eq:vtilde}
        \end{align}
        Summing \eqref{eq:vtilde} over all elements $K \in \mesh$, we obtain the result.
		\item ($k=n-1$)
        We have to prove that $\|\dive\Pi^{n-1}_hv\|\lesssim \|v\|_1$. We proceed by making use of the commuting property from \Cref{lem:co-chain projection} and \Cref{rem: Pi^n is L2},
        \begin{align}
            \|\dive\Pi^{n-1}_hv\|=\|\Pi^{n}_h\dive v\|\leq \|\dive v\| \leq \|v\|_1.
        \end{align}
    	\item ($k = n$) Since $dv=0$, the result follows directly from \Cref{cor:pot_stability}.
    \end{itemize}
\end{proof}

\subsection{Clément interpolation operator for virtual elements}\label{sec:cl_interpop}

We continue by constructing an additional interpolation operator that maps to the nodal spaces $W_h^k$. Let $\Theta^0_h: H^1(\Omega)\to V^0_h$ denote the Clément type interpolation operator introduced in \cite{Cangiani2017Posteriori}. Let now $\Theta^k_h$ be the operator acting on $v\in W^k = [H^1(\Omega)]^{\binom{n}{k}}$ by applying $\Theta^0_h$ component-wise. Then
\begin{align}
	\Theta^k_h: W^k \to W_h^k
\end{align} 
and its properties are shown in the following lemma.

\begin{lemma}[Clément interpolant]\label{lem:clem_interpol}
	Boundary conditions are preserved in the sense that $\Theta^k_h [H^1_0(\Omega)]^{\binom{n}{k}}\subset [H^1_0(\Omega)]^{\binom{n}{k}}$ and 
	\begin{align}
        \|(I-\Theta^k_h)v\| &\lesssim h\|v\|_1,
        & \|\Theta^k_hv\|_1 &\lesssim \|v\|_1, 
        &\forall v &\in W^k. \label{eq:Clem_k_continuity_stability}
	\end{align} 
\end{lemma}
% \begin{lemma}[Clément interpolant]\label{lem:clem_interpol}
%     Let $v\in H^1(\Omega).$ 
% 	There exists an operator $\Theta^0_h: H^1(\Omega)\to V^0_h$ for which
%     \begin{align}
%         \|(I-\Theta^0_h)v\|+h| (I-\Theta^0_h)v|_1 &\lesssim h\|v\|_1,& \Theta^0_h H^1_0(\Omega)&\subset H^1_0(\Omega). \label{eq:Clem_from_paper}
%     \end{align}
% 	It follows that
% 	\begin{align}
%         \|(I-\Theta^0_h)v\| &\lesssim h\|v\|_1,& \|\Theta^0_hv\|_1 &\lesssim \|v\|_1. \label{eq:Clem_continuity_stability}
%     \end{align} 
% \end{lemma}
\begin{proof}
	We prove the lemma for $\Theta^0_h$, since the proof for $\Theta^k_h$ follows by applying the arguments componentwise. 
    The following properties were shown in \cite[Thm. 11]{Cangiani2017Posteriori}:
	\begin{align}
        \|(I-\Theta^0_h)v\|+h| (I-\Theta^0_h)v|_1 &\lesssim h\|v\|_1,& \Theta^0_h H^1_0(\Omega)&\subset H^1_0(\Omega). \label{eq:Clem_from_paper}
    \end{align}
	This immediately provides the interpolation estimate $\|(I-\Theta^0_h) v\| \lesssim h\|v\|_1$. Moreover, \eqref{eq:Clem_from_paper} implies $|(I-\Theta^0_h)v|_1 \lesssim \|v\|_1$ and combined with a triangle inequality, we thus obtain stability:
    \begin{align}
        \|\Theta^0_h v\|_1 \lesssim \|v\|_1 + \|(I-\Theta^0_h) v\|_1  \lesssim \|v\|_1.
    \end{align}

\end{proof}

\subsection{The main result}\label{sec: Main result}

With the results presented in the previous subsections, we are now ready to prove the main result of this paper, namely the regular decomposition of VE spaces of lowest order.

\begin{theorem}[Virtual element regular decomposition]\label{thm:discrete_decomp}
	Given $v_h\in V_h^k$, there exist functions $\tilde{v}_h\in V_h^k,\ \psi_h\in W_h^k,$ and $p_h\in V^{k-1}_h$ such that 
	\begin{subequations}
		\begin{align}
			v_h &= \tilde{v}_h+\Pi^k_h\psi_h+dp_h. \label{eq: mainresult1}
		\end{align}
	Moreover, the decomposition is stable in the following sense:
		\begin{align}
			\|h^{-1}\tilde{v}_h\|+\|\psi_h\|_1 &\lesssim \|dv_h\|, &
			\|p_h\|_{V^{k - 1}} &\lesssim \|v_h\|_{V^k}. \label{eq: mainresult2}
		\end{align}
	\end{subequations}
\end{theorem}
\begin{proof}
	We follow \cite{HiptXu07}. First, we use the continuous regular decomposition of \Cref{lem:cont_decomp} to write 
	$v_h =\psi+dp$ with $\psi \in W^k$ and $p \in V^{k - 1}$ satisfying
	\begin{subequations}
		\begin{align}
			\|\psi\|_1&\lesssim \|dv_h\|, & \|p\|_{V^{k - 1}} &\lesssim \|v_h\|_{V^k}. \label{eq:psi_and_p_bound}
		\end{align}
	\end{subequations}
	Note that $d\psi = d(v_h - dp) = d v_h \in V^{k+1}_h$ and so $\psi\in X^k$. By \Cref{lem:domain Pi},  $\Pi^k_h\psi$ is then well-defined. Using the co-chain projection property from \Cref{lem:co-chain projection}, it follows that 
	\begin{align}
		d(I-\Pi^k_h)\psi = (I-\Pi^{k + 1}_h)d\psi = (I-\Pi^{k + 1}_h)dv_h = 0.
	\end{align}
	
	Now the exactness of the continuous complex implies the existence of a $q\in V^{k-1}$ such that $dq = (I-\Pi^k_h)\psi$.
	Let us consider the following decomposition of $\psi$:
	\begin{align}
		\psi = \underbrace{\Pi^k_h(\psi-\Theta^k_h\psi)}_{\in V_h^k} + \Pi^k_h\underbrace{\Theta^k_h\psi}_{\in W_h^k} + \underbrace{(I-\Pi^k_h)\psi}_{=dq\ \in \ \ker_d(V^k)}.
	\end{align}
	
	We set $\tilde{v}_h:=\Pi^k_h(\psi-\Theta^k_h\psi)$ and $\psi_h:=\Theta^k_h\psi.$ In order to define $p_h$, we first make the following observation:
	\begin{align}
		d(p+q) = v_h-\psi + (I-\Pi^k_h)\psi = v_h-\Pi^k_h\psi \in V_h^k.
	\end{align}
	\Cref{cor: stable_pot_h} now provides $p_h\in V^{k-1}_h$ with $dp_h=d(p+q)$ and $\|p_h\|_{V^k}\lesssim \|d(p+q)\|$.
	
	The decomposition \eqref{eq: mainresult1} now follows by construction, since
	\begin{align}
		\tilde{v}_h + \Pi^k_h\psi_h + dp_h = \Pi^k_h(\psi-\Theta^k_h\psi) + \Pi^k_h\Theta^k_h\psi + v_h-\Pi^k_h\psi = v_h.
	\end{align}

	We continue by proving the bounds for each component.
	First, by the approximation estimates of \Cref{thm:pot_interpol} and \Cref{lem:clem_interpol}, we derive
	\begin{subequations} \label{eq:mainresultproof}
	\begin{align}
		\|h^{-1}\tilde{v}_h\| &= \|h^{-1}\Pi^k_h(\psi-\Theta^k_h\psi)\| 
		\le \|h^{-1}(I-\Theta^k_h)\psi\| + \|h^{-1}(I-\Pi^k_h)\underbrace{(I-\Theta^k_h)\psi}_{\in X^k+W^k_h}\| \nonumber \\
		&\lesssim \|\psi\|_1 + \|(I-\Theta^k_h)\psi\|_1 \le 2\|\psi\|_1 + \|\Theta^k_h\psi\|_1
		\lesssim \|\psi\|_1 \lesssim \|dv_h\|,
	\end{align}
	in which the second-to-last inequality is due to the stability of $\Theta^k_h$ (\Cref{lem:clem_interpol}), and the last follows from \eqref{eq:psi_and_p_bound}. These last inequalities also allow us to bound 
	\begin{align}
		\|\psi_h\| = \|\Theta^k_h\psi\| \lesssim \|\psi\|_1 \lesssim \|dv_h\|.
	\end{align}
	Finally, the definition of $p_h$ combined with \Cref{thm:pot_interpol} and \eqref{eq:psi_and_p_bound} results in
	\begin{align}
		\|p_h\|_{V^{k - 1}} &
		\lesssim \|d(p+q)\| 
		\le \|dp\|+\|dq\| = \|dp\|+\|(I-\Pi^k_h)\psi\| 
		\lesssim \|dp\| + h\|\psi\|_1 
		\lesssim \|v_h\|_{V^k}.
	\end{align}
	\end{subequations}
	Together, \eqref{eq:mainresultproof} forms \eqref{eq: mainresult2}.
\end{proof}

By applying the regular decomposition from \Cref{thm:discrete_decomp} twice, we obtain a further decomposition that consists solely of \emph{regular} functions in the nodal spaces, e.g. $\psi_h \in W_h^k$, and \emph{high-frequency} functions, e.g. $\tilde v_h \in V_h^k$ for which $\| h^{-1} \tilde v_h \|$ is bounded.

\begin{corollary}[Deeper decomposition]\label{cor:deep_decomp}
	Given $v_h\in V_h^k$, there exist functions $\tilde{v}_h\in V^k_h,\ \psi_h\in W^k_h,$ and $\tilde{p}_h\in V^{k-1}_h,\ \phi_h\in W^{k-1}_h$ such that 
	\begin{subequations}
		\begin{align} \label{eq: deeper1}
			v_h &= \tilde{v}_h+\Pi^k_h\psi_h+d \tilde{p}_h+d\Pi^{k-1}_h\phi_h,
		\end{align}
	and
		\begin{align} \label{eq: deeper2}
			\|h^{-1}\tilde{v}_h\|+\|\psi_h\|_1 &\lesssim \|dv_h\|, &
			\|h^{-1}\tilde{p}_h\|+\|\phi_h\|_1%\lesssim \|p_h\|+\|d^{k-1}p_h\| 
			\lesssim \|v_h\|_{V^k}.
		\end{align}
	\end{subequations}
\end{corollary}
\begin{proof}
	%If $p_h\notin V^0_h$, which only happens for the case $(n=3;\ k=2)$, we can apply \Cref{thm:discrete_decomp} again on $p_h=\tilde{p}_h+\Pi^{k-1}_h\phi_h+d^{k-2}q_h$. (Otherwise if $k=1$ there is no decomposition of $p_h$ since $p_h=\psi+0p=\psi$ from \Cref{lem:cont_decomp}.)
	We first apply \Cref{thm:discrete_decomp} to decompose $v_h = \tilde{v}_h+\Pi^k_h\psi_h+dp_h$. We now apply \Cref{thm:discrete_decomp} again on $p_h \in V_h^{k - 1}$ to get 
	\begin{align}
		p_h=\tilde{p}_h+\Pi^{k-1}_h\phi_h+dq_h,
	\end{align}
	where $\tilde{p}_h\in V^{k-1}_h, \phi_h \in W^{k-1}_h$ and $q_h\in V^{k-2}_h$. 
	Substituting the definitions and using $dd = 0$, we obtain
	\begin{align}
		v_h&=\tilde{v}_h+\Pi^k_h\psi_h+d(\tilde{p}_h+\Pi^{k-1}_h\phi_h+dq_h)\nonumber\\
		&=\tilde{v}_h+\Pi^k_h\psi_h+d\tilde{p}_h+d\Pi^{k-1}_h\phi_h,
	\end{align} 
	which is exactly \eqref{eq: deeper1}. The bounds on $\tilde v_h$ and $\psi_h$ now follow from \Cref{thm:discrete_decomp}. Similarly, we obtain:
	\begin{align}
		\|h^{-1}\tilde{p}_h\|+\|\phi_h\|_1
		\lesssim \|dp_h\|  
		% \lesssim \|p_h\|+\|dp_h\| 
		\lesssim \|v_h\|_{V^k}.
	\end{align}
\end{proof}
\section{Auxiliary space preconditioning}\label{sec:auxprecond}

We now aim to use the regular decomposition from \Cref{sec: Main result} to construct an efficient numerical solver using the framework of auxiliary space preconditioning \cite{HiptXu07}. After introducing the framework in \Cref{sub: Abstract framework}, we apply it to the VE spaces in \Cref{sub: Virtual element auxiliary space preconditioners}. Finally, \Cref{sub: multiplicative} briefly shows a multiplicative variant of the preconditioner.

\subsection{Abstract framework}
\label{sub: Abstract framework}
In this exposition, we closely follow \cite{HiptXu07}. First, let $V$ be a Hilbert space with inner product $(\cdot, \cdot)_V$, norm $\| \cdot \|_V$ and dual space $V'$. Let $A: V \to V'$ be the linear operator associated to the inner product, i.e.
\begin{align} \label{eq: A inner product}
    \langle Au, v \rangle_{V' \times V} &= (u, v)_V, & 
    \langle Av, v \rangle_{V' \times V} &= \| v \|_V^2, &
    \forall u,v &\in V.
\end{align}

Second, let $S: V \to V'$ be a symmetric, positive definite, linear operator. In turn, $S$ induces an additional norm on $V$:
\begin{align}
    \| v \|_S &:= \langle S v, v \rangle_{V' \times V}^\frac12 &
    \forall v &\in V.
\end{align}
We assume that $S$ is easily invertible and refer to the application of $S^{-1}$ as the \emph{smoother}.

Third, let the \emph{auxiliary spaces} $W_j$ be Hilbert spaces for $j = 1, 2, \ldots, J$. 
Analogous to \eqref{eq: A inner product}, we let $A_j: W_j \to W_j'$ be the linear operator associated with the inner product on $W_j$.
We furthermore equip each $W_j$ with a continuous and linear \emph{transfer operator} $\pi_j: W_j\to V$, whose adjoint is denoted by $\pi_j^*: V' \to W_j'$.

% We consider the following problem: given $f\in V'$, find $u\in V$ such that
% \begin{equation}\label{eq:Axb}
%     Au=f. 
% \end{equation}

Then the abstract auxiliary space preconditioner $B:V'\to V$ is defined as:
\begin{equation} \label{eq:abstract_auxprecond}
    B := S^{-1}+ \sum_{j=1}^J \pi_j A^{-1}_j \pi_j^* .
\end{equation}

\begin{theorem}[Auxiliary space preconditioner]\label{thm:auxprecond}
Let $v\in V, w_j\in W_j$, $j=1,\dots, J$. 
If the following three conditions hold:
\begin{enumerate}
    \item The smoother $S^{-1}$ is continuous: 
    \begin{subequations}
    \begin{align}
        \| v \|_V &\le c_s \|v\|_S, &
        c_s &> 0.
    \end{align}
    \item The transfer functions $\pi_j$ are continuous for all $j = 1, \ldots, J$:
    \begin{align}
        \| \pi_j w_j \|_V &\le c_j \|w_j\|_{W_j}, &
        c_j &> 0.
    \end{align}
    % \item $\| \pi_j u_j \|_V \le c_j \left(\|u_j\|^2_{a}+\sum_{j=1}^J \|u_j\|_{U_j}^2\right)^{1/2},\ j=1,\dots,J.$
    \item For every $v \in V$, there exist $v_0 \in V$ and $w_j \in W_j$ such that
    \begin{align}
        v &= v_0 + \sum_{j = 1}^J \pi_j w_j, &
        \|v_0\|_S^2 + \sum_{j=1}^J \|w_j\|_{W_j}^2 &\le c_0 \|v\|_V^2, &
        c_0 &> 0.
    \end{align}
    \end{subequations} 
\end{enumerate}
Then, for $B$ given by \eqref{eq:abstract_auxprecond}, the spectral condition number of the preconditioned system $BA$ is bounded;
    \begin{align}
        \kappa (BA) := \frac{\lambda_{\max}(BA)}{\lambda_{\min}(BA)} \le c_0^2 \left(c_s^2+ \sum_{j=1}^Jc_j^2\right).
    \end{align}
\end{theorem}
\begin{proof}
    See the proof of \cite[Thm. 2.2]{HiptXu07} and the remarks around \cite[(2.6) and (2.14)]{HiptXu07}.
\end{proof}

\subsection{Virtual element auxiliary space preconditioners}
\label{sub: Virtual element auxiliary space preconditioners}

Before we construct the nodal auxiliary space preconditioners according to \Cref{thm:auxprecond}, we note that preconditioning the operator $A$ on $V_h^k$ is not necessary for all $k$. In particular:
\begin{itemize}
    \item If $k=0$, the decomposition of \Cref{thm:auxprecond} is trivial since $W_h^0 = V_h^0$ and so we may set $\psi_h = v_h$. In this case a multi-grid method is applicable, see e.g. \cite{antonietti2018multigrid}.
    \item For $k=n$, we have $V_h^n = \mathbb{P}_0$ and so the operator $A$ is diagonal and no preconditioner is necessary.
\end{itemize} 
We therefore focus on the remaining cases of $0 < k < n$. \Cref{sub: smoother} introduces the smoother and the particular constructions for $k = 1, 2$ are considered in \Cref{sub: precond edge,sub: precond face}, respectively.

% We can mimic \eqref{eq:Axb} and state \eqref{eq:model_problem} abstractly as 
% \begin{align}
% 	(I+(d^k)^*d^k)v_h = f, 
% \end{align}
% where $I+(d^k)^*d^k$ is a symmetric positive definite operator from $V^k_h$ to itself. (Note that here $d^k=\dive$ for $(n=2,3;k=n-1)$ or $d^k=\crl$ for $(n=3;k=1)$.) The induced norm is
% \begin{align}
%     \|v\|_{V^k} = (\|v\|^2+\|d^kv\|^2)^{1/2},\quad v\in V^k_h.
% \end{align}

% For the remainder, to keep a concise notation, let 
% \begin{align}
%     \tilde{p}_h &= \begin{cases}
%         \tilde{p}_h,\ &(n,k)=(3,2) \\
%         p_h,\ &else
%     \end{cases}, &
%     \phi_h &= \begin{cases}
%         \phi_h,\ &(n,k)=(3,2) \\
%         0,\ &else
%     \end{cases}.
% \end{align}

\subsubsection{A virtual element smoother} \label{sub: smoother}

Due to the discussion in the previous subsection, we only need to define the smoother for $0 < k < n$. We propose two alternatives. First, given \ekn{$v_h\in V_h^k$}, let $\sum_i v_{h, i}$ be its decomposition such that each $v_{h, i}$ is a scaled basis function of $V_h^k$. We then define the diagonal operator
\begin{align}
    \| v_h \|_{D^k}^2 
    &= \langle D^k v_h,v_h \rangle 
    = \sum_i \| v_{h, i} \|_{V^k}^2. \label{eq: diag smoother}
\end{align}
Second, recalling \Cref{lem:VEM_stability_ops} we propose using the stabilization term common to VEM:
\begin{subequations}
\begin{align}
    \| v_h \|_{S^k}^2 
    &= \langle S^k v_h,v_h \rangle 
    = \sum_{K\in\mesh} \| h^{-1} v_h \|_{S^k, K}^2, \label{eq: smoother}\\
    \| v_h \|_{S^k, K}^2 
    &= \begin{cases}
        % (v_h,v_h)_K,\ &(k=n) \\
        h\sum_{F\in\mcF(K)} (v_h\cdot\nv,v_h\cdot\nv)_F,\ &(k=n-1) \\
        h^2 \sum_{E\in\mcE(K)} (v_h\cdot\tv,v_h\cdot\tv)_E,\ &(k=n-2;n=3) %\\
        % h^3 I_K
    \end{cases} \label{eq: mass stabilization}
\end{align}
\end{subequations}
Note that \eqref{eq: mass stabilization} corresponds to the VEM stabilization operator defined in \cite[(54),(58),(73)]{DaVeiga2022LowestInterpop}. 
Recall that the decompositions from \Cref{thm:discrete_decomp} and \Cref{cor:deep_decomp} include high-frequency terms that are bounded in norms weighted by $h^{-1}$. In order to make the connection between these terms and the smoother \eqref{eq: smoother}, we present the following lemma. 

\begin{lemma}[Two smoothers]\label{lem: smoother}
    For $0 < k < n$, the operators $D^k$ and $S^k$ induce norms on $V_h^k$ that satisfy
\begin{align}
    \| v_h \|_{V^k} &\lesssim
    \| v_h \|_{D^k} \lesssim
    \| v_h \|_{S^k}
    % &  \| v_h \|_{S^k} 
    \lesssim \| h^{-1} v_h \|, &
    \forall v_h &\in V_h^k.
\end{align}
\end{lemma}
\begin{proof}
We prove the three bounds separately.
\begin{enumerate}
    \item We follow \cite[(7.2)]{HiptXu07}. We have for any $v_h\in V^k_h$ that the square of its energy norm can be computed by summing squares of local contributions;
    \begin{align}
        \|v_h\|^2_{V^k} &= \sum_{K\in\mesh}\| \sum_i v_{h, i} \|^2_{V^k, K}
        \leq \sum_{K\in\mesh} N_K \sum_i \| v_{h, i} \|^2_{V^k, K}
        \leq \max_{K\in\mesh} N_K \|v_h\|_{D^k}^2,
    \end{align}
    by the Cauchy-Schwartz inequality. Here $N_K$ is the number of basis functions whose support overlaps with $K$.

    \item For the part involving only the $L^2$-norm of $v_h$, \Cref{lem:VEM_stability_ops} gives us
    \begin{align*}
        \sum_i \| v_{h, i} \|^2 
        \eqsim \sum_i \sum_{K \in \mesh} \| v_{h, i} \|_{S^k, K}^2
        = \sum_{K \in \mesh} \sum_i \| v_{h, i} \|_{S^k, K}^2
        = \sum_{K \in \mesh} \| v_h \|_{S^k, K}^2
        \lesssim \sum_{K \in \mesh} \| h^{-1} v_h \|_{S^k, K}^2
        = \| v_h \|_{S^k}^2
    \end{align*}

    It remains to show that $\sum_i \| dv_{h, i} \|^2 \lesssim \| v_h \|_{S^k}^2$. We distinguish two cases
    \begin{itemize}
        \item ($k = n - 1$) Let us rewrite $v_{h, i} = \alpha_i b_{h, i}$ with $\alpha_i \in \mathbb{R}$ and $b_{h, i}$ a basis function of $V_h^k$. On an element $K$ in the support of $b_{h, i}$, we compute:
        \begin{align}
            \| dv_{h, i} \|_K^2
            = \alpha_i^2 \| \dive b_{h, i} \|_K^2
            = \alpha_i^2 |K|^{-1}
            &\eqsim \alpha_i^2 h^{-1} |F_i|^{-1} \nonumber\\
            &= \alpha_i^2 h^{-1} \| \nv \cdot b_{h, i} \|_{F_i}^2
            = h^{-1} \| \nv \cdot v_{h, i} \|_{F_i}^2
            = h^{-1} \| \nv \cdot v_h \|_{F_i}^2
        \end{align}
        Summing over all basis functions and all elements, we obtain
        \begin{align}
            \sum_i \| dv_{h, i} \|^2 
            \eqsim \sum_i h^{-1} \| \nv \cdot v_h \|_{F_i}^2
            \lesssim \sum_{K \in \mesh} h^{-2} \| v_h \|_{S^k, K}^2
            = \| v_h \|_{S^k}^2
        \end{align}
        \item ($k = n - 2$, $n=3$) Let again $v_{h, i} = \alpha_i b_{h, i}$. Then 
        \begin{align} \label{eq: n-2 a}
            \| dv_{h, i} \|_K^2
            = \alpha_i^2 \| \crl b_{h, i} \|_K^2
            \eqsim \alpha_i^2 \| \crl b_{h, i} \|_{S^{n - 1}, K}^2
            &= \sum_{F \in \mcF(E_i)} h\alpha_i^2 \| \nv \cdot \crl b_{h, i} \|_{F}^2 \nonumber \\
            &= \sum_{F \in \mcF(E_i)} h\alpha_i^2 \| \dive_F (\nv \times b_{h, i}) \|_{F}^2
        \end{align}

        From the definition of the edge space $V_h^{n - 2}$, it follows that $\dive_F (\nv \times b_{h, i}) \in \mathbb{P}_0(F)$. In particular, we note that
        \begin{align}
            \dive_F (\nv \times b_{h, i})
            = \frac{1}{|F|} \int_F \dive_F (\nv \times b_{h, i})
            = \frac{1}{|F|} \int_{\partial F} \tv \cdot b_{h, i}
            = \frac{1}{|F|} \int_{E_i} \tv \cdot b_{h, i}
            = \frac{\pm 1}{|F|}
        \end{align}

        Substituting this in \eqref{eq: n-2 a} leads to
        \begin{align}
            \| dv_{h, i} \|_K^2
            &\eqsim \sum_{F \in \mcF(E_i)} h \alpha_i^2 |F|^{-1}
            \lesssim \alpha_i^2 |E_i|^{-1}
            = \alpha_i^2 \| \tv \cdot b_{h, i} \|_{E_i}^2
            = \| \tv \cdot v_{h, i} \|_{E_i}^2
            = h^{-2} \| v_{h, i} \|_{S^k, K}^2
        \end{align}

        Finally, we sum over all elements and basis functions to obtain
        \begin{align}
            \sum_i \| dv_{h, i} \|^2 
            \lesssim \sum_{K \in \mesh} h^{-2} \| v_{h, i} \|_{S^k, K}^2
            = \sum_{K \in \mesh} h^{-2} \| v_h \|_{S^k, K}^2
            = \| v_h \|_{S^k}^2
        \end{align}
    \end{itemize}
    \item The final inequality follows from \Cref{lem:VEM_stability_ops}.
    \begin{align}
        \| v_h \|_{S^k}^2
        = \sum_{K\in\mesh} \| h^{-1} v_h \|_{S^k, K}^2
        \eqsim \sum_{K\in\mesh} \| h^{-1} v_h \|_K^2
        = \| h^{-1} v_h \|^2
    \end{align}
\end{enumerate}
\end{proof}

\Cref{lem: smoother} shows that we have a choice between two smoothers, either $D^k$ or $S^k$. Moreover, it implies inverse inequalities for the differential operators curl and div applied to the edge respectively facet functions. We summarize this result in the following corollary.

\begin{corollary}[Edge and facet inverse inequalities] \label{cor: inverse inequalities}
    For the facet and edge virtual element spaces of lowest order, i.e. with $0<k<n$, the following inverse inequality holds:
    \begin{align}
        \| dv_h \| &\lesssim \| h^{-1} v_h \|, &
        \forall v_h &\in V_h^k.
    \end{align}
\end{corollary}
\begin{proof}
    The result follows from \Cref{lem: smoother}, since $\|dv_h\|\leq \|v_h\|_{V^k}.$
\end{proof}

% Define now the smoother $S^{-1}$ as the inverse of the diagonalisation of $I+(d^k)^*d^k$, seen as a linear operator between finite dimensional spaces 
% \begin{align}
% 	S^{-1}:= \diag(I+(d^k)^*d^k)^{-1} : V^k_h\to V^k_h.
% \end{align}
% The matrix representation of the finite dimensional operator $S^{-1}$ is the inverse of the diagonal part of the matrix representation of $I+(d^k)^*d^k$. 
% This choice of $S$ is called Jacobi smoothing and we have for all $v_h\in V^k_h$ that
% \begin{align}
%     \|v_h\|_{S^k}^2 = \sum_{j=1}^{\dim V^k_h} |c_j|^2\|\varphi_j\|^2_{V^k},
% \end{align}
% where $\sum_{j=1}^{\dim V^k_h} c_j\varphi_j$ is the basis function decomposition of $v_h$.
% \begin{lemma}\label{lem: smoother}
%     Property 1 of \Cref{thm:auxprecond} is satisfied with constant $c_s$ equal to the maximum number of basis functions with support on some element.
% \end{lemma}
% \begin{proof}
% \end{proof}

\subsubsection{An auxiliary space preconditioner for edge elements}
\label{sub: precond edge}

The first case we consider is $k = 1$ with $n = 3$.
Given $v_h \in V_h^1$, \Cref{thm:discrete_decomp} provides the following decomposition:
\begin{align}
    v_h &= \tilde{v}_h + \Pi_h^0\psi_h + \grad p_h, &
    \|h^{-1}\tilde{v}_h\|^2+\|\psi_h\|_1^2 + \|p_h\|_1^2 &\lesssim \|v_h\|_{V^1}^2.
\end{align}
Thus, we are led to define the following auxiliary spaces with corresponding inner products and transfer operators:
\begin{subequations}
\begin{align}
    W_1 &:= W_h^1, & 
    \langle A_1 \psi_h, \psi_h \rangle &= \| \psi_h \|_1^2, &
    \pi_1 &:= \Pi_h^1 : W_h^1 \to V_h^1, \\
    W_2 &:= W_h^0 = V_h^0, & 
    \langle A_2 p_h, p_h \rangle &= \| p_h \|_1^2, &
    \pi_2 &:= \grad : V_h^0 \to V_h^1.
\end{align}
\end{subequations}

\begin{theorem}\label{thm:auxprecond_edge}
    Let $A$ denote the operator associated with the inner product of $V_h^1$ and let
    \begin{align}
        B = \left( S^1 \right)^{-1} 
        + \Pi_h^1 A_1^{-1} \left(\Pi_h^1\right)^*
        + \grad A_2^{-1} \grad^*
    \end{align}
    then the spectral condition number $\kappa(BA)$ is bounded, independent of the mesh size $h$.
\end{theorem}
\begin{proof}
    We verify the three conditions of \Cref{thm:auxprecond}.
    \begin{enumerate}
    \item Continuity of the smoother was shown in \Cref{lem: smoother}.
    \item The transfer operator $\Pi_h^1: W_h^1 \to V_h^1$ is continuous due to \Cref{cor:Wk_stab_in_Vk}. Since the second transfer operator is a differential $d: V_h^{k-1} \to V_h^k$, its continuity follows from the fact that $dd = 0$:
    \begin{align} \label{eq: continuity d}
        \| d p_h\|_{V^k} &= \| d p_h\| \le \| p_h\|_{V^{k - 1}}, &
        \forall p_h &\in V_h^{k - 1}.
    \end{align}
    \item The decomposition is given by \Cref{thm:discrete_decomp}. For the stability of the decomposition, we only need to show that $\| \tilde{v}_h \|_{S^1} \lesssim \| h^{-1} \tilde{v}_h \|$, which was proven in \Cref{lem: smoother} (last inequality).
    \end{enumerate}
    The robustness with respect to $h$ follows from the fact that the constants $c_i$ with $i \in \{s, 0, 1, 2\}$ are independent of $h$.
\end{proof}

\begin{remark}[Two-dimensional facet elements] \label{rem: 2D preconditioner}
    The only difference for $n = 2$ and $k = 1$, is that $\Pi_2$ needs to be redefined as $\crl: V_h^0 \to V_h^1$. With this minor adjustment, the analogue of \Cref{thm:auxprecond} follows in the same manner.
\end{remark}

\subsubsection{An auxiliary space preconditioner for facet elements in 3D}
\label{sub: precond face}
Given $v_h \in V_h^2$ with $n = 3$, \Cref{cor:deep_decomp} provides the following decomposition:
\begin{align}
    v_h &= \tilde{v}_h + \Pi_h^2 \psi_h + \crl \tilde{p}_h + \crl \Pi_h^1 \phi_h, &
    \|h^{-1}\tilde{v}_h\|^2+\|\psi_h\|_1^2 + \| h^{-1} \tilde{p}_h\|^2 + \| \phi_h \|_1^2 &\lesssim \|v_h\|_{V^2}^2.
\end{align}
Thus, we are led to define the following auxiliary spaces, inner products and transfer operators:
\begin{subequations}
\begin{align}
    W_1 &:= W_h^2, & 
    \langle A_1 \psi_h, \psi_h \rangle &= \| \psi_h \|_1^2, &
    \pi_1 &:= \Pi_h^2 : W_h^2 \to V_h^2, \\
    W_2 &:= V_h^1, & 
    \langle A_2 \tilde{p}_h, \tilde{p}_h \rangle &= \| \tilde{p}_h \|_{S^1}^2, &
    \pi_2 &:= \crl : V_h^1 \to V_h^2 \\
    W_3 &:= W_h^1, & 
    \langle A_3 \phi_h, \phi_h \rangle &= \| \phi_h \|_1^2, &
    \pi_3 &:= \crl \Pi_h^1 : W_h^1 \to V_h^2.
\end{align}
\end{subequations}

\begin{theorem}\label{thm:auxprecond_facet}
    For $n = 3$, let $A$ denote the operator associated with the inner product of $V_h^2$ and let
    \begin{align}
        B = \left( S^2 \right)^{-1} 
        + \Pi_h^2 A_1^{-1} \left(\Pi_h^2\right)^*
        + \crl \left( (S^1)^{-1} + \Pi_h^1 A_3^{-1} \left(\Pi_h^1\right)^* \right) \crl^*
    \end{align}
    then the spectral condition number $\kappa(BA)$ is bounded, independent of the mesh size $h$.
\end{theorem}
\begin{proof}
    We verify the three conditions of \Cref{thm:auxprecond}.
    \begin{enumerate}
    \item Continuity of the smoother $S^2$ was shown in \Cref{lem: smoother}.
    \item The transfer operators $\Pi_h^k: W_h^k \to V_h^k$ are continuous due to \Cref{cor:Wk_stab_in_Vk}, so $\Pi_1$ is continuous. For $\Pi_2$, we use the continuity of the $\crl$ \eqref{eq: continuity d} combined with \Cref{lem: smoother} to obtain
    \begin{align}
        \| \Pi_2 \tilde{p}_h \|_{V^2} 
        = \| \crl \tilde{p}_h \|_{V^2} = \| \crl \tilde{p}_h \| 
        \le \| \tilde{p}_h \|_{V^1} 
        \lesssim \| \tilde{p}_h \|_{S^1} 
    \end{align}
    Finally, the continuity of $\Pi_3$ follows by combining \Cref{cor:Wk_stab_in_Vk} with \eqref{eq: continuity d}. 
    \item The decomposition is given by \Cref{cor:deep_decomp}. Again, the stability of the decomposition follows from $\| \tilde{v}_h \|_{S^2} \lesssim \| h^{-1} \tilde{v}_h \|$, shown in \Cref{lem: smoother}.
    \end{enumerate}
    Since the constants $c_i$ with $i \in \{s, 0, 1, 2, 3\}$ are independent of $h$, the preconditioner is robust.
\end{proof}

\subsection{A multiplicative variant}
\label{sub: multiplicative}

The preconditioners in \Cref{thm:auxprecond_edge} and \Cref{thm:auxprecond_facet} can be seen as \emph{additive} preconditioners. A \emph{multiplicative} variant of these operators can be constructed as follows.
For a given residual $r_0$, we first apply the smoother to create $z_0 \in V_h^k$:
\begin{subequations}
\begin{align}
    z_0 &= S^{-1}r_0.
\end{align}
Then for each $j \in \{1, \ldots, J\}$, we apply the following steps sequentially:
\begin{align}
    r_j &= r_{j - 1} - Az_{j - 1} \\
    z_j &= z_{j - 1} + \pi_j A_j^{-1} \pi_j^* r_j
\end{align}
\end{subequations}

Recall that $A$ is the operator from \eqref{eq: A inner product}. 
Finally, the multiplicative variant of the preconditioner outputs $z_J \in V_h^k$. 
\section{Numerical experiments}\label{sec:numex}

In this section, we present numerical experiments to illustrate the performance of auxiliary space preconditioning for virtual element discretizations of the $H^{\dive}$ projection problem \eqref{eq:divdiv} and the Darcy problem \eqref{eq:darcy} in 2D. 

The experiments are performed in the open source Julia-based virtual element library VirtuE and the run scripts are available at \cite{repo}. The numerical experiments are performed on a laptop with an Intel Core i7-8565U CPU and 16 GB of RAM. Throughout our tests we will use $D^k$ of \Cref{lem: smoother} as the smoother in the auxiliary space preconditioner. In our experiments, its performance was better in all cases compared to the smoother $S^k$ obtained from the stabilization term. We will refer to the additive preconditioner from \Cref{rem: 2D preconditioner} as $B_{\add}$ and the multiplicative variant from \Cref{sub: multiplicative} as $B_{\mult}$.

For each problem, we compare the condition numbers and the performance of GMRES (without restarts) for the original problem $Av=b$ against the preconditioned system $BAv=Bb$. 
We will conduct two types of numerical experiments. The first tests the robustness of the preconditioner with respect to the mesh size derived in \Cref{sec:auxprecond}, whereas the second ventures outside the theory to investigate the dependence of the preconditioner on element aspect ratios. 

\subsection{\texorpdfstring{$H^{\dive}$}{Hdiv} projection problem}
\label{sub: Hdiv projection problem}

Let us start by describing the VE discretization of problem \eqref{eq:divdiv}: find $v_h\in V^{n-1}_h$ such that 
\begin{align}\label{eq:divdiv_h}
	(v_h,w_h)_h+(\dive v_h,\dive w_h) &= (f,w_h), & 
    \forall w_h&\in V^{n-1}_h.
\end{align}
Here, $(v_h,w_h)_h$ is the mass matrix of the facet space, assembled according to \cite{Brezzi2014VEMMixedPrinciples}.
% In matrix form, this problem can be written as $Av=b$ where $A$ is the matrix representation of the bilinear form $(\cdot,\cdot)+(\dive\cdot,\dive\cdot)$, assembled according to standard VEM principles such as those found in \cite{Brezzi2014VEMMixedPrinciples}, and $b$ is the vector representation of the linear form $(f,\cdot)$. 
%(The system matrices $A$, $B_{\add}A$ and $B_{\mult}A$ are symmetric so we recall that for a symmetric matrix $C$, $\kappa(C)=\frac{\max\lambda(C)}{\min\lambda(C)}$, where $\lambda(C)$ is the set of eigenvalues of $C$.) 
%
The function $f$ on the right-hand side is given by
\begin{align}\label{eq:divdiv_data1}
    f(x, y) = - \begin{bmatrix}
        2\pi \cos(2\pi x)  \sin(4\pi  y) \\
        4\pi  \cos(4\pi  y) \sin(2\pi x)
    \end{bmatrix}.
\end{align}
The right-hand side is assembled by first projecting $w_h$ to the piecewise linears.
We construct the mesh by subdividing a structured, $N \times N$ Cartesian grid into triangles, squares, and pentagons, as illustrated in \Cref{fig:diamondmesh_2D}.

\begin{figure}[htbp]
    \centering
    \includegraphics[%
        width=0.49\textwidth]{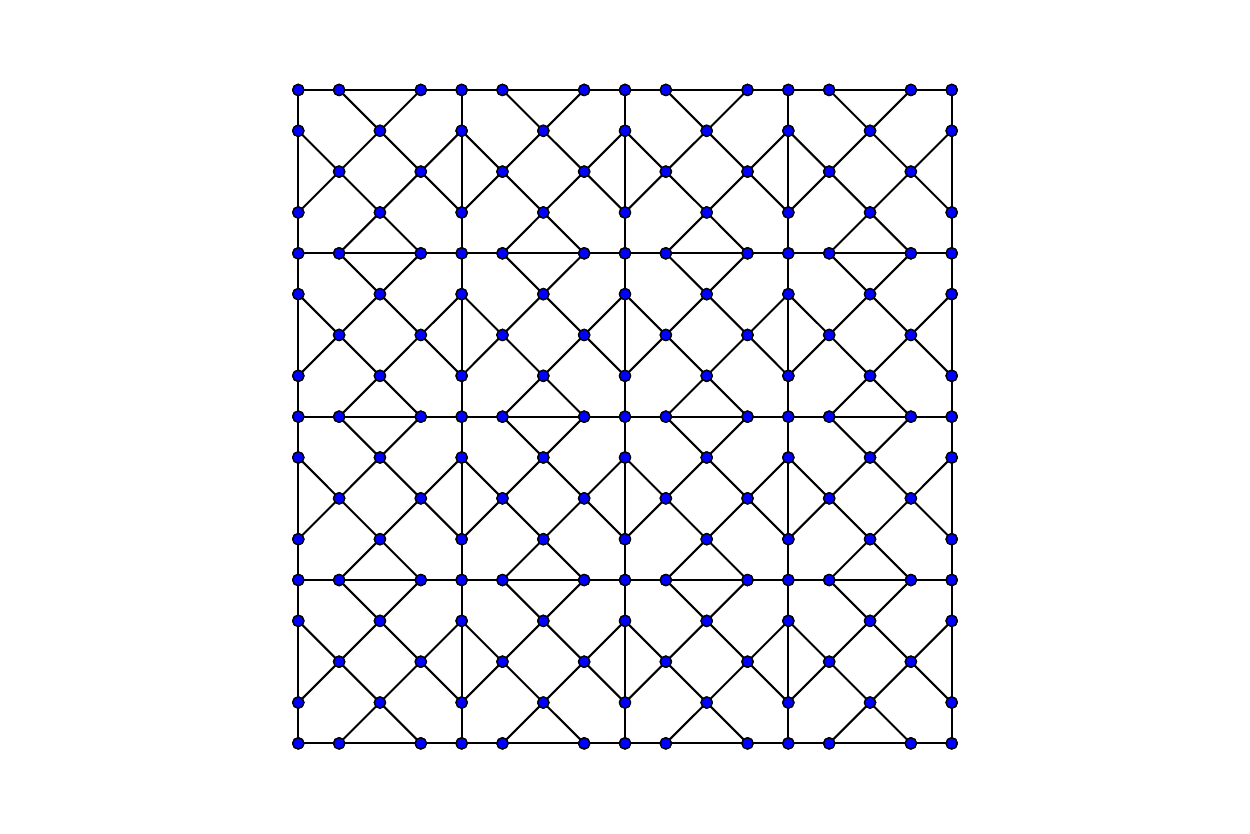}%
        \label{fig:diamondmesh_2D_1}
    \hfill
    \includegraphics[%
        width=0.49\textwidth]{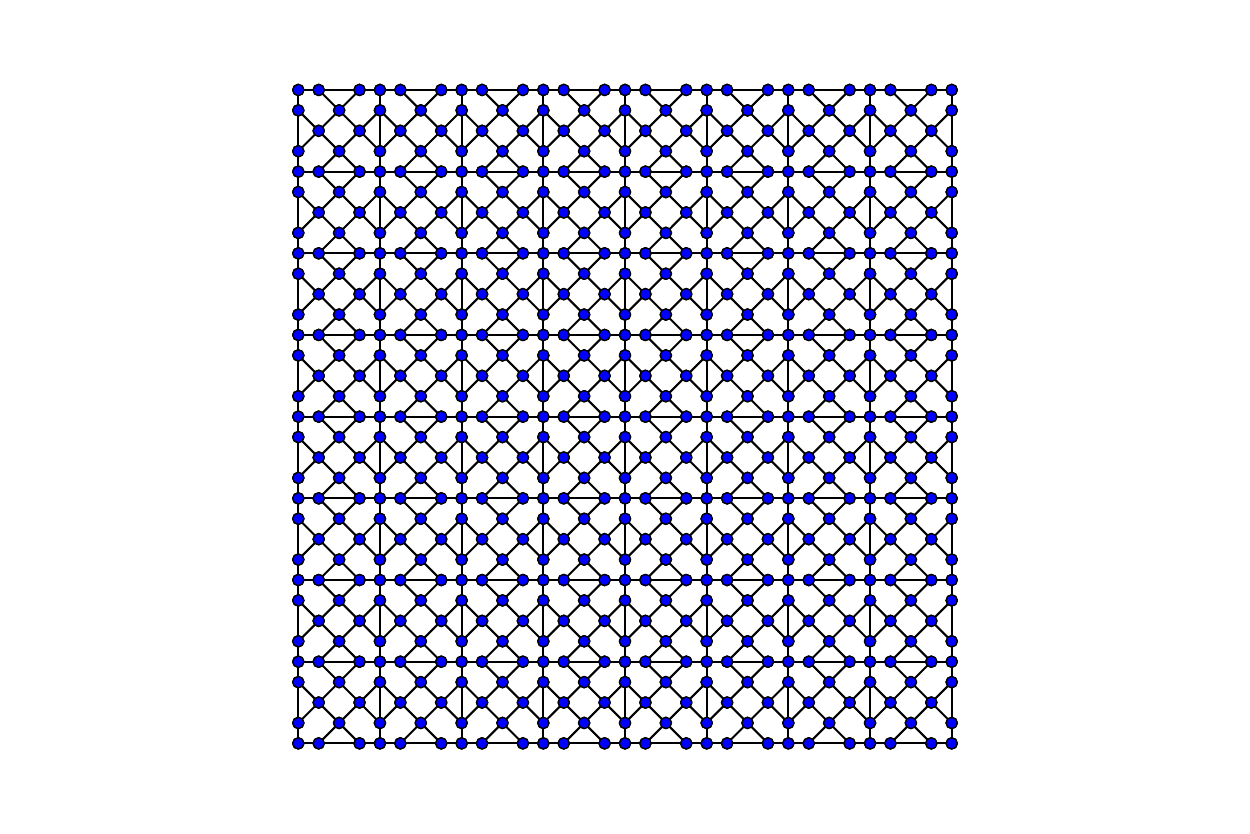}%
    \label{fig:diamondmesh_2D_11}
    \caption{The two coarsest meshes (with $N=4,8$ respectively) used in the mesh size dependency experiments of \Cref{sub: Hdiv projection problem,sub: Darcy problem}. This sequence of meshes satisfies \Cref{ass:mesh}.}
    \label{fig:diamondmesh_2D}
\end{figure}

The results are shown in \Cref{tab:divdiv_meshsize}, in which $A$ denotes the matrix representation of the left-hand side of \eqref{eq:divdiv_h}. We see that both in terms of conditions numbers and GMRES iterations, the auxiliary space preconditioned system has essentially constant values, and the smallest values. Most importantly, the number of GMRES iterations is not growing with decreasing mesh size, for both the additive and multiplicative variants. 
% 
% \begin{table}[htbp]
%     \centering
%     \resizebox{\textwidth}{!}{
%     \begin{tabular}{c|c|c|c|c|c|c|c}
%         $N$ & $\max_{K\in\mesh}\frac{\text{diam}(K)^2}{|K|}$ & $\kappa(M)$ & $\kappa(\diag(M)^{-1} M)$ & $\kappa(BM)$ & GMRES on $M$ & GMRES on $\diag(M)^{-1} M$ & GMRES on $BM$ \\
%         \hline
%         5  & 4 & 4.99E03  & 5.25E03  & 4.49 & 194 & 180 &  16 \\
%         10 & 4 & 2.00E04  & 2.10E04  & 4.50 & 402 & 365 &  16 \\ %605858
%         20 & 4 & 8.02E04  & 8.42E04  & 4.51 & 798 & 735 &  16 \\ %6072121
%     \end{tabular} }
%     \caption{Comparison of the condition numbers and number of iterations for GMRES on the original and auxiliary space preconditioned systems for the $H^{\dive}$ projection problem under mesh size increments.}
%     \label{tab:divdiv_meshsize}
% \end{table}
\begin{table}[htbp]
    \centering
    %\resizebox{\textwidth}{!}{
    \caption{Performance of the nodal auxiliary space preconditioner on the $H^{\dive}$ projection problem for varying mesh sizes.}
    \begin{tabular}{cc|cccc|cccc}
            &          & \multicolumn{4}{c|}{Condition numbers}                                                           & \multicolumn{4}{c}{GMRES iterations}                                    \\ %\cline{3-9} 
        $N$ &  \#dof & $A$         & $\diag(A)^{-1} A$         & $B_{\add}A$    & $B_{\mult}A$     & $A$ & $\diag(A)^{-1} A$ & $B_{\add}A$ & $B_{\mult}A$ \\ \hline
        4   &  312     & 3.19E03     & 3.36E03                   & 4.48    & 9.29     & 114 & 24                & 14   & 16    \\
        8   &  1200    & 1.28E04     & 1.35E04                   & 4.50    & 9.53     & 234 & 121               & 16   & 17    \\
        16  &  4704    & 5.13E04     & 5.39E04                   & 4.51    & 9.60     & 427 & 220               & 15   & 17    \\
        32  &  18624   & 2.05E05     & 2.15E05                   & 4.51    & 9.61     & 739 & 350               & 15   & 15    
    \end{tabular} %}
    \label{tab:divdiv_meshsize}
\end{table}

% face additive
% 23
% 25
% 25
% 24
% face multiplicative:
% 16
% 16
% 15
% 14

\subsection{Darcy problem}
\label{sub: Darcy problem}
The discrete Darcy problem reads as follows. Given data $f\in [L^2(\Omega)]^n$ and $g\in L^2(\Omega)$, find $u_h\in V^{n-1}_h$ and $p_h\in V^{n}_h$ such that
\begin{subequations}\label{eq:darcy_h}
\begin{align}
    (u_h,v_h)_h - (p_h,\dive v_h) &= (f,v_h), &\forall v_h&\in V^{n-1}_h, \\
    -(q_h,\dive u_h) &= (g,q_h), &\forall q_h&\in V^{n}_h.
\end{align}
\end{subequations}
We denote the corresponding system matrix by $\mcA$ and the mass matrix corresponding to $(u_h,v_h)_h$ by $M_u$. 
% This saddle-point problem has a system matrix $A$ of the form 
% \begin{align}
%     A:=\begin{bmatrix}
%         M & C^T \\
%         C & 0
%     \end{bmatrix}.
% \end{align}

Next, we note that \eqref{eq:darcy_h} is well-posed in the product space $H^{\dive}(\Omega) \times L^2(\Omega)$.
Following the framework of norm-equivalent preconditioning \cite{mardal2011preconditioning}, the Riesz representation operator forms a robust preconditioner for this problem. This involves solving an $H^{\dive}$ projection problem and a projection in $L^2$. We therefore construct a block diagonal preconditioner $\Bf$ where the first block is the nodal auxiliary space preconditioner $B$ and the second block is the inverse of the (diagonal) mass matrix $M_p$ on $\mathbb{P}_0$. 
% The preconditioned system matrix thus takes the form,
% \begin{align}
%     \BfA=\begin{bmatrix}
%         B & 0 \\
%         0 & M_p^{-1}
%     \end{bmatrix}
%     \begin{bmatrix}
%         M & C^T \\
%         C & 0
%     \end{bmatrix}= \begin{bmatrix}
%         BM & BC^T \\
%         M_p^{-1}C & 0
%     \end{bmatrix}.
% \end{align}
% (Notice that the resulting matrix is not symmetric.) 
We compare its performance to a naively constructed preconditioner composed of the diagonal of the mass matrices. Namely, we compare:
\begin{align}
    \Bf &:= \begin{bmatrix} B & 0 \\ 0 & M_p^{-1}\end{bmatrix}, & 
    \Bf_{\diag} &:= \left(\diag\begin{bmatrix} M_u & 0 \\ 0 & M_p\end{bmatrix}\right)^{-1}.
\end{align}

Let $f$ be as in \eqref{eq:divdiv_data1} and let $g=-40\pi^2 \cos(2\pi x)  \sin(4\pi  y).$ 
% \begin{align*}
%     %f&=-(2\pi \cos(2\pi x)  \sin(4\pi  y), 4\pi  \cos(4\pi  y) \sin(2\pi x)), \\
%     g&=40\pi^2 \cos(2\pi x)  \sin(4\pi  y).
% \end{align*}
We use the same sequence of meshes as depicted in \Cref{fig:diamondmesh_2D} and compare the %the condition numbers and 
numbers of iterations for GMRES on the original and preconditioned systems. The results are shown in \Cref{tab:darcy_meshsize}. The condition numbers of the last iteration are not computed. We note that the number of GMRES iterations is growing with mesh size for the original and diagonal preconditioned systems. In contrast, the auxiliary space preconditioned systems require a stable number of iterations, for both additive and multiplicative variants. 

\begin{table}[htbp]
    \centering
    %\resizebox{\textwidth}{!}{
    \caption{Performance of the nodal auxiliary space preconditioner on the Darcy problem for varying mesh sizes. Computation of the condition numbers on the finest grid was not feasible in our implementation and are therefore omitted.}
    \begin{tabular}{cc|cccc|cccc}
            &     & \multicolumn{4}{c|}{Condition number}       & \multicolumn{4}{c}{GMRES iterations}                                    \\ %\cline{3-5} 
        $N$ &  \#dof & $\mcA$ & $\Bf_{\diag} \mcA$ & $\Bf_{\add}\mcA$ & $\Bf_{\mult}\mcA$ & $\mcA$ & $\Bf_{\diag} \mcA$ & $\Bf_{\add}\mcA$ & $\Bf_{\mult}\mcA$ \\ \hline
        4   &  456       &  3.65E01 & 7.59E01   & 3.12   &   3.21      & 84   & 17    & 29  & 29  \\
        8   &  1776      &  1.29E02 & 1.50E02   & 3.14   &   3.27      & 136  & 34    & 34  & 32  \\
        16  &  7008      &  5.00E02 & 2.98E02   & 3.17   &   3.29      & 211  & 32    & 35  & 35  \\
        32  &  27840     &  -       & -         & -      &   -         & 384  & 53    & 35  & 35 
    \end{tabular} %}
    \label{tab:darcy_meshsize}
\end{table}

\subsection{Aspect ratio tests}
\label{sub: Aspect ratio tests}

In this next experiment, we consider a sequence of meshes on the unit square $\Omega=[0,1]^2$ induced by a parameter $\epsilon>0$. For each $\epsilon$, we intersect a background mesh with the line $y=0.5+\epsilon$. At each intersection point we create a new vertex and split the adjacent elements, see \Cref{fig:mesh_2D}. This example therefore corresponds to a naive remeshing procedure to conform a background mesh to an independently placed, embedded interface.
Strictly speaking, this example falls out of the scope of our theory, in particular \Cref{ass:mesh} is violated because the aspect ratios will increase as $\epsilon$ decreases.
We measure the (maximum) aspect ratio of the mesh $\mesh$ by
\begin{align} 
    \alpha:=\max_{K\in\mesh}\frac{\text{diam}(K)^2}{|K|}.
\end{align} 

\begin{figure}[htbp]
    \centering
    \includegraphics[width=0.49\textwidth]{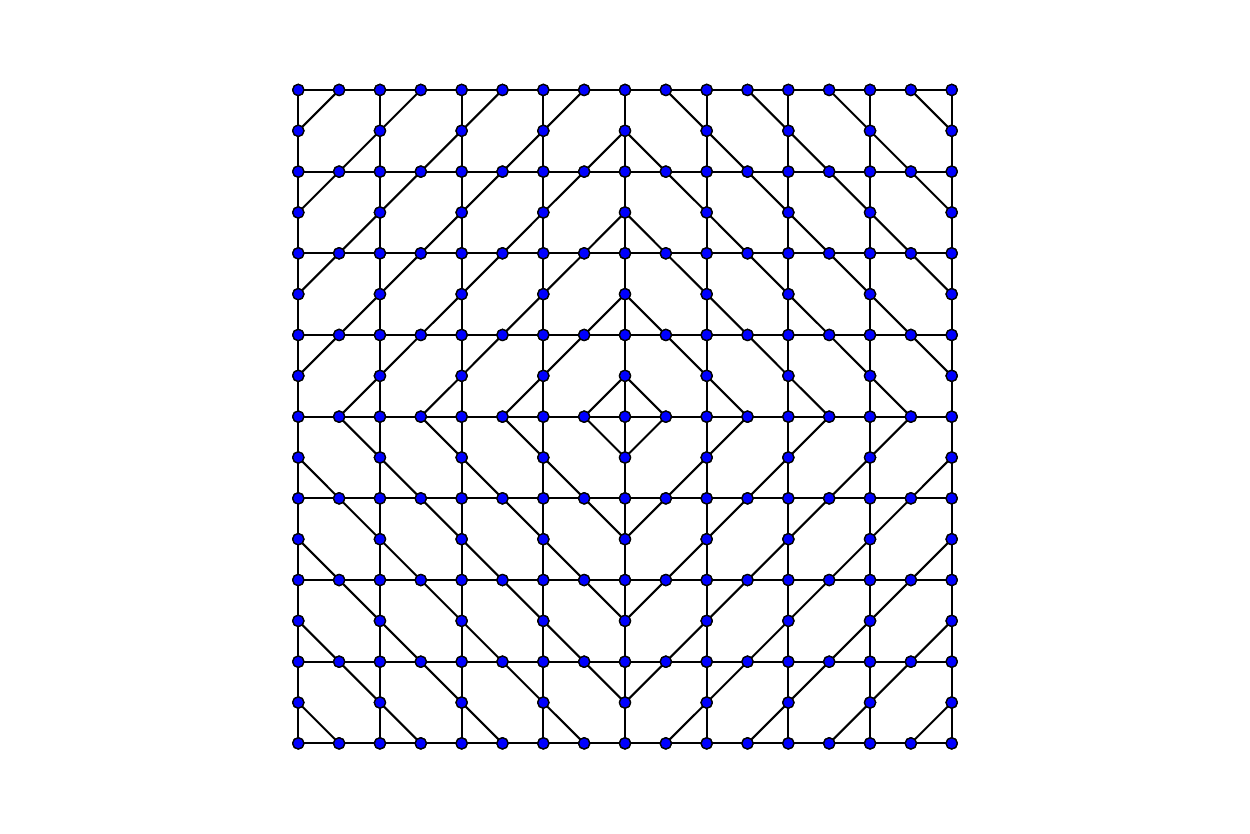}
    \hfill
    \includegraphics[width=0.49\textwidth]{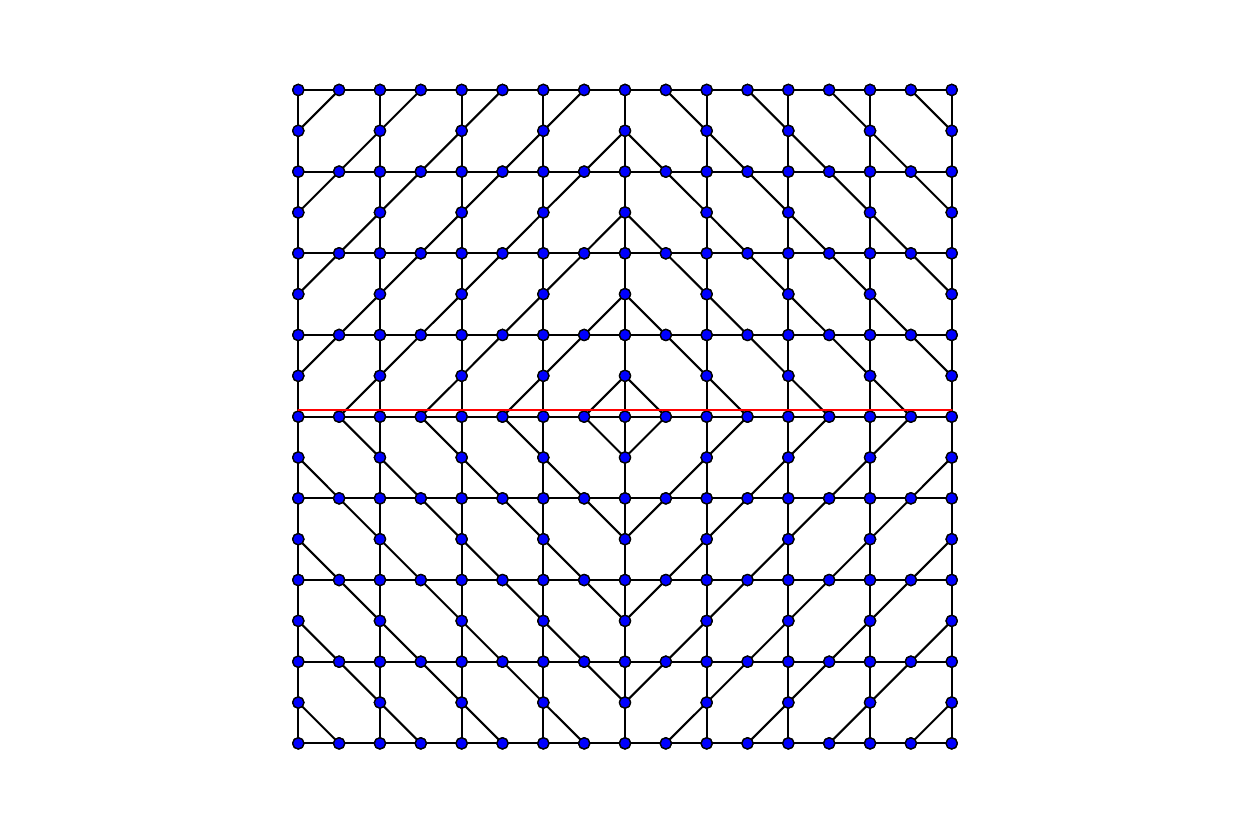}
    \caption{Meshes used in the numerical experiments for the aspect ratio tests. The elements crossed by the red line at $y = 0.5+\epsilon$ are split in two by introducing new nodes at the intersections with mesh edges. }%Left: Initial mesh with $\alpha=4$. Right: mesh with $\epsilon = 0.01$ and $\alpha=7.94$.}
    \label{fig:mesh_2D}
\end{figure}

\subsubsection{$H^{\dive}$ projection problem}

On the mesh depicted \ekn{in \Cref{fig:mesh_2D}} we solve the $H^{\dive}$ projection problem with 
\begin{align}\label{eq:divdiv_data2}
    f = \begin{bmatrix}
        \cos(x)\sinh(y) \\
        \sin(x)\cosh(y)
    \end{bmatrix},
\end{align} 
and compare the condition number and number of iterations for GMRES on the original and preconditioned systems.
The results are shown in \Cref{tab:divdiv}. Although the aspect ratio is increasing, the auxiliary space preconditioner performs well, for both additive or multiplicative variants. The number of GMRES iterations is stable and relatively small, and the same holds for the condition numbers.

\begin{table}[htbp]
    \centering
    %\resizebox{\textwidth}{!}{
    \caption{$H^{\dive}$ projection problem: Comparison of the condition numbers and number of iterations for GMRES on the original and preconditioned systems for meshes with increasing aspect ratios.}
    \begin{tabular}{cc|cccc|cccc}
        &                                                 & \multicolumn{4}{c|}{Condition numbers}                                                           & \multicolumn{4}{c}{GMRES iterations}                                    \\ %\cline{3-9} 
    $\epsilon$ &  $\alpha$ & $A$         & $\diag(A)^{-1} A$         & $B_{\add}A$    &  $B_{\mult}A$   & $A$ & $\diag(A)^{-1} A$ & $B_{\add}A$ & $B_{\mult}A$ \\ \hline
    -          &  4.00E00       & 5.18E03     & 3.62E03       & 4.33    & 11.5      & 126 & 116               & 16  &  18 \\
    1E-2       &  7.94E00       & 2.00E04     & 1.12E04       & 4.49    & 13.0      & 169 & 204               & 16  &  18 \\
    1E-4       &  6.27E02       & 1.91E06     & 7.74E05       & 4.93    & 11.3      & 179 & 270               & 16  &  18 \\
    1E-6       &  6.25E04       & 1.91E08     & 7.70E07       & 5.00    & 11.3      & 174 & 275               & 16  &  17 \\
    1e-8       &  6.25E06       & 1.91E10     & 7.70E09       & 5.00    & 11.3      & 177 & 159               & 16  &  17
    \end{tabular} %}
    \label{tab:divdiv} 
\end{table}

\subsubsection{Darcy problem}

Finally, we perform the same aspect ratio experiment for the Darcy problem \eqref{eq:darcy_h}. 
Let here $f$ be as in \eqref{eq:divdiv_data2} and let $g=0$. 
In \Cref{tab:darcy} we compare the results. As $\alpha$ increases, the number of GMRES iterations increases for the original and diagonal preconditioned systems, but remains constant for the auxiliary space preconditioned system. This is again true for both additive and multiplicative variants, indicating a promising potential of this approach for general polygonal meshes.

\begin{table}[htbp]
    \centering
    %\resizebox{\textwidth}{!}{
    \caption{Darcy problem: Comparison of the condition numbers and 
    number of iterations for GMRES on the original and preconditioned systems for meshes with increasing aspect ratios.}
    \begin{tabular}{cc|cccc|cccc}
        &  & \multicolumn{4}{c|}{Condition numbers}  & \multicolumn{4}{c}{GMRES iterations}                                    \\ %\cline{3-5} 
    $\epsilon$ &  $\alpha$    & $\mcA$ & $\Bf_{\diag}\mcA$ & $\Bf_{\add}\mcA$ & $\Bf_{\mult}\mcA$ & $\mcA$ & $\Bf_{\diag}\mcA$ & $\Bf_{\add}\mcA$ & $\Bf_{\mult}\mcA$ \\ \hline
    -          &  4.00E00     & 4.72E01 & 9.88E01 & 3.21  &  3.29     & 82  & 58    & 32  &  32 \\
    1E-2       &  7.94E00     & 8.17E01 & 1.54E02 & 3.21  &  3.29     & 117 & 167   & 36  &  36 \\
    1E-4       &  6.27E02     & 8.17E03 & 2.80E04 & 3.21  &  3.31     & 116 & 230   & 35  &  36 \\
    1E-6       &  6.25E04     & 8.17E05 & 2.75E07 & 3.21  &  3.31     & 111 & 176   & 35  &  36 \\
    1E-8       &  6.25E06     & 1.43E09 & 2.75E10 & 3.21  &  3.31     & 107 & 414   & 34  &  35
    \end{tabular} %}
    \label{tab:darcy} 
\end{table}
\section{Conclusion}\label{sec:conclusion}

In this paper, we have formulated and analyzed nodal auxiliary space preconditioners for edge and facet Virtual Element spaces. The preconditioner effectively consists of a series of elliptic problems on nodal VE spaces combined with appropriate smoothers. Our analysis shows that the performance of these preconditioners is independent of the mesh size, which is supported by numerical experiments for the $H^{\dive}$ projection problem and the Darcy problem, in 2D. 

Future research directions on this topic include the following. First, while the theory includes three-dimensional problems, the performance of the preconditioner needs to be verified numerically. Second, the theory can likely be extended to higher-order VE spaces, see e.g. \cite{VeiBeiBrezzMarRuss16}, by deriving a regular decomposition of the associated, exact co-chain complex and computing appropriate transfer operators from the higher-order $H^1$-conforming VE spaces. Third, the assumptions on the mesh may be weakened, in particular concerning the convexity of the elements, since the numerical experiments indicate a wider applicability. 

\ekn{Finally, we note that the proposed preconditioner involves solving global, nodal systems at each iteration. Therefore, our current implementation does not provide a convincing speed-up in terms of solution time for the model problems. Nevertheless, this preconditioner serves as a starting point for the development of alternative solvers in which the global solves are replaced by scalable, spectrally equivalent operators. These can be obtained, for example, by employing (algebraic) multi-grid methods \cite{Kolev2009,Kolev2012}. The computational efficiency of the resulting preconditioners forms a topic for future research.}

In conclusion, we have generalized the Hiptmair-Xu nodal auxiliary space preconditioner to the VEM framework. Our numerical results furthermore show promising potential of the preconditioner for problems on meshes with high aspect ratios, obtained by naively adapting a grid to conform to an embedded interface. 

\section*{Acknowledgments}
	The authors warmly thank Lorenzo Mascotto and Lourenco Beir\~{a}o da Veiga for valuable discussions concerning \Cref{lem:edge_interpol}, and Ana Budi\v{s}a for suggesting the multiplicative variant of the preconditioner (\Cref{sub: multiplicative}).
%\input{Todo}

% \paragraph{Acknowledgement} Wietses grant+did it in my spare time??%This research was supported by the Swedish Research Council Grant No. 2018-05262 and the Wallenberg Academy Fellowship KAW 2019.0190.

\bibliographystyle{siamplain}
\bibliography{references}

% \appendix
%\input{AppendixHighOrderVEM}
% \input{Regularity.tex}

\end{document}